\documentclass[11pt]{amsart}
\usepackage{amsmath, amssymb, latexsym}
\usepackage{mathrsfs}
\usepackage{enumerate}
\usepackage[all, knot]{xy}
\usepackage{color}

\newcommand{\nc}{\newcommand}
\newtheorem{Thm}{Theorem}[section]
\newtheorem{Prop}[Thm]{Proposition}
\newtheorem{Cor}[Thm]{Corollary}
\newtheorem{Lem}[Thm]{Lemma}

\theoremstyle{definition}
\newtheorem{defn}[Thm]{Definition}

\theoremstyle{remark}

\newenvironment{red}
{\relax\color{red}}
{\hspace*{.5ex}\relax}

\newcommand{\ber}{\begin{red}}
\newcommand{\er}{\end{red}}

\newenvironment{verd}
{\relax\color{magenta}}
{\hspace*{.5ex}\relax}

\newcommand{\bg}{\begin{verd}}
\newcommand{\eg}{\end{verd}}

\numberwithin{equation}{subsection}

\newcommand{\Z}{\mathbb{Z}}
\newcommand{\Q}{\mathbb{Q}}

\newcommand{\A}{\mathbb{A}}

\newcommand{\g}{\mathfrak{g}}

\newcommand{\Hom}{\mathrm{Hom}}

\newcommand{\wt}{{\rm wt}}

\newcommand{\proj}{\mathrm{Proj}}
\newcommand{\fmod}{\mathrm{Mod}}

\newcommand{\rlQ}{\mathsf{Q}}   
\newcommand{\wlP}{\mathsf{P}}   
\newcommand{\cmA}{\mathsf{A}}  
\newcommand{\tyX}{\mathsf{X}}  

\newcommand{\ST}{\mathsf{ST}}   
\newcommand{\res}{\mathrm{res}}   
\newcommand{\sg}{\mathfrak{S}}   

\newcommand{\fqh}{R^{\Lambda_0}}   
\newcommand{\df}{\mathsf{def}} 
\newcommand{\bR}{\mathbf{k}} 
\newcommand{\fqH}{R^{\Lambda_0}} 
\newcommand{\sh}{\mathsf{sh}} 
\newcommand{\dg}{\mathsf{deg}} 
\newcommand{\cdg}{\mathsf{codeg}} 

\newcommand{\qE}{\mathfrak{E}}
\newcommand{\qF}{\mathfrak{F}}
\newcommand{\qm}{\mathfrak{m}}
\newcommand{\hst}{\varpi}

\nc{\groundwall}
{\xy
(0,-0.5)*{};(33,-0.5)*{} **\dir{.};
(0,-1)*{};(33,-1)*{} **\dir{.};
(0,-1.5)*{};(33,-1.5)*{} **\dir{.};
(0,-2)*{};(33,-2)*{} **\dir{.};
(0,-2.5)*{};(33,-2.5)*{} **\dir{.};
(0,-3)*{};(33,-3)*{} **\dir{-};
(0,0)*{};(33,0)*{} **\dir{-};
(33,-3)*{};(33,0)*{} **\dir{-};
(27,-3)*{};(27,0)*{} **\dir{-};
(21,-3)*{};(21,0)*{} **\dir{-};
(15,-3)*{};(15,0)*{} **\dir{-};
(9,-3)*{};(9,0)*{} **\dir{-};
(3,-3)*{};(3,0)*{} **\dir{-};
(30.2,-1.5)*{_0}; (24.2,-1.5)*{_0}; (18.2,-1.5)*{_0};
(12.2,-1.5)*{_0}; (6.2,-1.5)*{_0};
\endxy}

\begin{document}

\title[Young walls and graded dimension formulas]
{Young walls and graded dimension formulas for finite quiver Hecke
algebras of type $A^{(2)}_{2\ell}$ and $D^{(2)}_{\ell+1}$}

\author[Se-jin Oh]{Se-jin Oh$^{1}$}
\thanks{$^1$ This work was supported by BK21 PLUS SNU Mathematical Sciences Division} 
\address{Department of Mathematical Sciences, Seoul National University GwanakRo 1, Gwanak-Gu, Seoul 151-747, Korea}
\email{sejin092@gmail.com}

\author[Euiyong Park]{Euiyong Park$^{2}$}
\thanks{$^2$ This work was supported by the 2013 Research Fund of the University of Seoul.} 
\address{Department of Mathematics, University of Seoul, Seoul 130-743, Korea}
\email{epark@uos.ac.kr}


\begin{abstract}
We study graded dimension formulas for finite quiver Hecke algebras
$R^{\Lambda_0}(\beta)$ of type $A^{(2)}_{2\ell}$ and
$D^{(2)}_{\ell+1}$ using combinatorics of Young walls. We introduce
the notion of standard tableaux for proper Young walls and show that
the standard tableaux form a graded poset with lattice structure. We
next investigate Laurent polynomials associated with proper Young walls and
their standard tableaux arising from the Fock space representations
consisting of proper Young walls. Then we prove the graded dimension
formulas described in terms of the Laurent polynomials. When
evaluating at $q=1$, the graded dimension formulas recover the
dimension formulas for $R^{\Lambda_0}(\beta)$ described in terms of
standard tableaux of strict partitions.
\end{abstract}

\maketitle


\vskip 1em

\section*{Introduction}

The quiver Hecke algebras $R=\bigoplus_{\beta \in
\rlQ^{+}}R(\beta)$, introduced by Khovanov-Lauda \cite{KL09, KL11}
and Rouquier \cite{R08}, provide categorifications of the negative
parts of quantum Kac-Moody algebras. They have $\Z$-grading and can
be defined for all symmetrizable Cartan datums. Moreover, their
certain graded quotients $R^\Lambda=\bigoplus_{\beta \in
\rlQ^{+}}R^\Lambda(\beta)$, called cyclotomic quiver Hecke algebras,
categorify integrable highest weight modules $V(\Lambda)$ over
quantum Kac-Moody algebras \cite{KK11}. When $\Lambda = \Lambda_0$,
we call the algebras $\fqH(\beta)$ \emph{finite quiver Hecke
algebras}. The results in \cite{BK09,R08} tell that the block
algebras of the cyclotomic Hecke algebras are isomorphic to
cyclotomic quiver Hecke algebras of type $A_{\ell}^{(1)}$.
In particular, when $\Lambda = \Lambda_0$,
the finite quiver Hecke algebra of type $A_\ell^{(1)}$ is isomorphic to the Iwahori-Hecke algebra associated with the symmetric group.
In this sense, finite quiver Hecke algebras can be understood as vast
generalizations of Iwahori-Hecke algebras. Via
the isomorphism in \cite{BK09,R08}, well-established theories on
cyclotomic Hecke algebras can be applicable to $R^\Lambda$ of type
$A^{(1)}_\ell$ with additional grading structures
\cite{BK09A,BKW11,HM10}. However, only a few results about the
algebraic structure of $R^\Lambda$ for affine types except
$A^{(1)}_\ell$ were known \cite{AP12,AP13}. For example, there exist
no known homogeneous bases for $R^\Lambda$ of affine types except
type $A^{(1)}_\ell$.


The study of partitions (or Young diagrams) is one of the most
important research areas in combinatorial representation theory.
Combinatorics of partitions are deeply related to representations of
Lie algebras and symmetric groups. Among proper subsets of
partitions, strict partitions (or shifted Young diagrams) have its
own meaning in combinatorial representation theory.
In particular,
they realize the Fock space representations of Kac-Moody algebras of
types $A_{2\ell}^{(2)}$ and $D_{\ell+1}^{(2)}$ \cite{DJKM82,JM83}.
In \cite{Kang03}, Kang introduced new combinatorial objects, called {\it
Young walls}, to give combinatorial realizations of level $1$
crystals of various quantum affine Kac-Moody algebras,
which are deeply related to the theory of perfect crystals \cite{KKMMNN91, KKMMNN92}.
In \cite{KMPY96}, Kashiwara, Miwa, Pertersen and Yung constructed the Fock space representation of quantum affine algebras using perfect crystals.
In \cite{KangKwon08}, Kang and Kwon proved that the spaces
spanned by {\it proper} Young walls indeed form the Fock space
representations of quantum affine Kac-Moody algebras. In particular, proper Young walls of type
$A^{(2)}_{2\ell}$ (resp.\ $D^{(2)}_{\ell+1}$) can be identified with
$(2\ell+1)$-strict partitions (resp. $(\ell+1)$-strict partitions)
which are defined as follows: For $h \in \Z_{\ge 1} \sqcup \{ \infty
\}$,
\begin{center}
$\lambda=(\lambda_0,\lambda_1,\lambda_2,\ldots)$ is an $h$-strict partition if $\lambda_i=\lambda_{i+1} \ne 0$ implies $h | \lambda_i$.
\end{center}
An $h$-strict partition can be viewed as {\it intermediate} one between a partition and a strict partition since a partition is a $1$-strict partition
and a strict partition is an $\infty$-partition. Moreover, there exist the following inclusions:
$$ \{ \text{strict partitions} \} \subsetneq \cdots \subsetneq \{ \text{$2h$-strict partitions} \} \subsetneq \{ \text{$h$-strict partitions} \}
\subsetneq \{ \text{partitions} \}. $$ Thus, the classical limit of
the Fock space representation in \cite{KangKwon08} consisting of
proper Young walls of type $A^{(2)}_{2\ell}$ and $D^{(2)}_{\ell+1}$
is {\it strictly larger} than the Fock space representation in
\cite{DJKM82,JM83} of consisting of strict partitions. The
differences and connections for type $A_{2\ell}^{(2)}$ were
discussed in \cite[Appendix D]{KMPY96}, \cite[Section 2]{LT97}.
It is explained in \cite[Introduction]{Tsu10} why proper Young walls of types $A_\ell^{(1)}$, $A_{2\ell}^{(2)}$ and $D_{\ell+1}^{(2)}$ may be regarded as partitions in the viewpoint of perfect crystals.
Type $A^{(2)}_{2\ell}$ and $D^{(2)}_{\ell+1}$ also appear in categorification of the Hecke-Clifford superalgebras \cite{BK01, Tsu10} and the quiver Hecke superalgebras \cite{HW12,KKO12}.

This paper aims for computing {\it graded} dimension formulas of $R^{\Lambda_0}(\beta)$ of type  $A^{(2)}_{2\ell}$
and $D^{(2)}_{\ell+1}$ using combinatorics of proper Young walls.
To achieve the goal, we introduce and investigate new combinatorial
objects, {\it standard tableaux of proper Young walls}. We define a
standard tableau $T$ of a proper Young wall $Y$ as a sequences of
certain proper Young walls, and also give an equivalent definition
for standard tableaux by  considering its entries (Proposition
\ref{Prop: standard}). As $h$-strict partitions do, standard
tableaux of a proper Young wall can be understood as intermediate
one between standard tableaux of a partition and those of a strict
partition.
As a combinatorial object itself, the set $\ST(Y)$ of standard tableaux on a proper Young wall $Y$ has a partial order induced from the weak order of the symmetric group,
which is an analogue of the partial order appeared in \cite[Section 3.3]{BKW11}.
It turns out that $w (\ST(Y) )$ is a {\it generalized quotient} which is introduced in \cite{BW88} for generalizing distinguished coset representatives of parabolic subgroups of Coxeter groups
(Theorem \ref{Thm: weak order}).
Theorem \ref{Thm: weak order} tells that $\ST(Y)$ is a graded poset with lattice structure having a unique minimum and a unique maximum and that, for $T, T' \in \ST(Y)$ with $T<T'$,
there exist simple transpositions $ s_{i_1}, s_{i_2}, \ldots, s_{i_t}$ such that,
$$ T' = s_{i_t} \cdots s_{i_1} T, \quad s_{i_j} \cdots s_{i_1} T \in \ST(Y) \text{ for all $1\le j \le t$}, $$
where $w T$ is the tableau obtained from $T$ by replacing entry $k$ with $w(k)$ for $k=1,\ldots,n$.

Standard tableaux for proper Young walls naturally arise from the
actions $f_{i_1} \cdots f_{i_n}$ and $e_{j_1} \cdots e_{j_n}$ on the
Fock space representations $\mathcal{F}(\Lambda_0)$ consisting of
proper Young walls. Observing the action of Chevalley generators
$f_i$ and $e_i$ on $\mathcal{F}(\Lambda_0)$, we get the Laurent
polynomials $\qE_q(T)$ and $\qF_q(T)$ associated with standard
tableaux $T$ of proper Young walls $Y$. They may be viewed as
$A_{2\ell}^{(2)}$-type and $D_{\ell+1}^{(2)}$-type analogues of
$q^{deg(T)}$ and $q^{-codeg(T)}$ which appear in graded dimension
formulas for finite quiver Hecke algebra of type $A_\ell^{(1)}$.
Here, ${deg(T)}$ and ${codeg(T)}$ are degree and codegree for
standard tableaux appeared in \cite[Section 3.5]{BKW11}.
Interestingly, the ratio of $\qE_q(T)$ and $\qF_q(T)$ becomes a
Laurent polynomial again, which depends only on $Y$ (Corollary
\ref{Cor: ratio E and F}). For any pair of standard tableaux
$(T,T')$ of $Y$ with $T<T'$ and $s_i T = T'$, we compute
$\qE_q(T')/\qE_q(T)$ and $\qF_q(T')/\qF_q(T)$ explicitly and show
that they are same (Proposition \ref{Prop: ratio for E, F}).
Moreover, Proposition \ref{Prop: m(Y)} tells that $\qm_q(T):=
\qE_q(T)/\qE_q(T^Y)=\qF_q(T)/\qF_q(T^Y)$ is a Laurent polynomial,
where $T^Y$ is the canonical tableau defined in $\eqref{Eq:
canonical tableau}$.

Our graded dimension formulas for $\fqh(\beta)$ are described in
terms of the Laurent polynomials associated with standard tableaux
for proper Young walls. Considering the Fock space representations
$\mathcal{F}(\Lambda_0)$ with the categorification theorem, we prove
the graded dimension formulas for $\fqh(\beta)$ (Theorem \ref{Thm: q-dim
formulas}). Our graded dimension formulas are graded versions of the
dimension formulas given in \cite{AP12, AP13} described in terms of
standard tableaux on strict partitions. When evaluating at $q=1$,
the dimension of $R^{\Lambda_0}(\beta)$ can be computed from a
proper subset $\ST_\infty(Y)$ of the standard tableaux $\ST(Y)$
which corresponds bijectively to the set of standard tableaux for
strict partitions (Theorem \ref{Thm: dim formulas}). Thus our graded
dimension formulas recover the dimension formulas in
\cite{AP12,AP13} described in terms of standard tableaux of strict
partitions.

\vskip 1em

\section{Fock space representations} \label{Sec: Sec Fock space}


\subsection{Quantum affine algebras} \label{Sec: quantum affine algs}
Let $I = \{0,1, \ldots, \ell \} \ (\ell \ge 2) $ be an index set. The affine Cartan matrices of type $A_{2\ell}^{(2)}$ and $ D_{\ell+1}^{(2)}$ are given as follow:
$$ \tiny
 \cmA (A_{2\ell}^{(2)}) =  \left(
                                  \begin{array}{ccccccc}
                                    2  & -2 & 0  & \ldots & 0  & 0  & 0 \\
                                    -1 &  2 & -1 & \ldots & 0  & 0  & 0 \\
                                    0  & -1 & 2  & \ldots & 0  & 0  & 0 \\
                                    \vdots   &  \vdots  &  \vdots  & \ddots &  \vdots  &  \vdots  & \vdots \\
                                    0  & 0  & 0  & \ldots & 2  & -1 & 0 \\
                                    0  & 0  & 0  & \ldots & -1 & 2  & -2 \\
                                    0  & 0  & 0  & \ldots & 0  & -1 & 2 \\
                                  \end{array}
                                \right),
\quad \cmA(D_{\ell+1}^{(2)}) = \left(
                                  \begin{array}{ccccccc}
                                    2  & -2 & 0  & \ldots & 0  & 0  & 0 \\
                                    -1 &  2 & -1 & \ldots & 0  & 0  & 0 \\
                                    0  & -1 & 2  & \ldots & 0  & 0  & 0 \\
                                    \vdots   &  \vdots  &  \vdots  & \ddots &  \vdots  &  \vdots  & \vdots \\
                                    0  & 0  & 0  & \ldots & 2  & -1 & 0 \\
                                    0  & 0  & 0  & \ldots & -1 & 2  & -1 \\
                                    0  & 0  & 0  & \ldots & 0  & -2 & 2 \\
                                  \end{array}
                                \right).
  $$
Throughout this paper, we set $\tyX = A_{2\ell}^{(2)}$ or $ D_{\ell+1}^{(2)}$, and let $\cmA (\tyX)=(a_{ij})_{i,j\in I}$ denote the affine Cartan matrix of type $\tyX$.

A realization of the affine Cartan matrix $\cmA (\tyX)$ is given as follows:
\begin{itemize}
\item[(1)] $\wlP$ is a free abelian group of rank $\ell+2$, called the {\it weight lattice},
\item[(2)] $\Pi = \{ \alpha_i \mid i\in I \} \subset \wlP$, called the set of {\it simple roots},
\item[(3)] $\Pi^{\vee} = \{ h_i \mid i\in I\} \subset \wlP^{\vee} := \Hom( \wlP, \Z )$, called the set of {\it simple coroots},
\end{itemize}
which satisfy
\begin{itemize}
\item[(a)] $\langle h_i, \alpha_j \rangle  = a_{ij}$ for all $i,j\in I$,
\item[(b)] $\Pi$ and $\Pi^{\vee}$ are linearly independent sets.
\end{itemize}
There is a symmetric bilinear pairing $(\ \mid\ )$ on $\wlP$ satisfying
$$ ( \alpha_i \mid \Lambda ) = \mathsf{d}_i \langle h_i , \Lambda \rangle \ \text{ for all } \Lambda \in \wlP,$$
where
\begin{align} \label{Eq: diagonal matrix}
(\mathsf{d}_0,\mathsf{d}_1, \ldots, \mathsf{d}_\ell) = \left\{
                                                            \begin{array}{ll}
                                                              (1,2, \ldots, 2,4) & \hbox{ if } \tyX = A_{2\ell}^{(2)}, \\
                                                              (1,2, \ldots,2,1) & \hbox{ if } \tyX = D_{\ell+1}^{(2)}.
                                                            \end{array}
                                                          \right.
\end{align}
We denote by $\wlP^{+} =
\{ \Lambda \in \wlP \mid \langle h_i, \Lambda \rangle \ge 0,\ i\in I
\}$ the set of {\it dominant integral weight}. For $i\in I$, let
$\Lambda_i$ denote the $i$th {\it fundamental weight}, i.e. $
\langle h_j, \Lambda_i \rangle = \delta_{ij}$ for $j\in I$. We call
$\rlQ = \bigoplus_{i \in I} \Z \alpha_i$ the {\it root lattice}, and
$\rlQ^+ = \sum_{i\in I} \Z_{\ge 0} \alpha_i$ the {\it positive cone}
of the root lattice. For $\beta=\sum_{i \in I} k_i \alpha_i \in
\rlQ^+$, the {\it height} of $\beta$ is defined by $|\beta|=\sum_{i
\in I} k_i$.


Let $q$ be an indeterminate and $m,n \in \Z_{\ge 0}$. For $i\in I$,
let $q_i = q^{ \mathsf{d}_i}$ and
\begin{equation*}
 \begin{aligned}
 \ &[n]_i =\frac{ q^n_{i} - q^{-n}_{i} }{ q_{i} - q^{-1}_{i} },
 \ &[n]_i! = \prod^{n}_{k=1} [k]_i ,
 \ &\left[\begin{matrix}m \\ n\\ \end{matrix} \right]_i=  \frac{ [m]_i! }{[m-n]_i! [n]_i! }.
 \end{aligned}
\end{equation*}

\begin{defn} \label{Def: GKM}
The {\em quantum affine algebra} $U_q(\tyX)$ associated
with $(\mathsf{A}(\tyX),\mathsf{P},\Pi,\Pi^{\vee})$ is the associative
algebra over $\Q(q)$ with $1$ generated by $e_i,f_i$ $(i \in I)$ and
$q^{h}$ $(h \in \mathsf{P}^{\vee})$ satisfying following relations:
\begin{enumerate}
  \item  $q^0=1, q^{h} q^{h'}=q^{h+h'} $ for $ h,h' \in \mathsf{P}^{\vee},$
  \item  $q^{h}e_i q^{-h}= q^{\langle h, \alpha_i \rangle} e_i,
          \ q^{h}f_i q^{-h} = q^{-\langle h, \alpha_i \rangle }f_i$ for $h \in \mathsf{P}^{\vee}, i \in I$,
  \item  $e_if_j - f_je_i =  \delta_{ij} \dfrac{K_i -K^{-1}_i}{q_i- q^{-1}_i }, \ \ \mbox{ where } K_i=q_i^{ h_i},$
  \item  $\displaystyle \sum^{1-a_{ij}}_{k=0} (-1)^k \left[\begin{matrix}1-a_{ij} \\ k\\ \end{matrix} \right]_i e^{1-a_{ij}-k}_i
         e_j e^{k}_i = 
  \sum^{1-a_{ij}}_{k=0} (-1)^k \left[\begin{matrix}1-a_{ij} \\ k\\ \end{matrix} \right]_i f^{1-a_{ij}-k}_if_j
        f^{k}_i=0 \quad \text{ for }  i \ne j. $
\end{enumerate}
\end{defn}

Let $\A = \Z[q,q^{-1}]$. We denote by $U_\A(\tyX)$ the $\A$-form of $U_q(\tyX)$. For a dominant integral weight $\Lambda \in \mathsf{P}^+$,
let $V(\Lambda)$ be the irreducible highest weight $U_q(\tyX)$-module with highest weight $\Lambda$
and $V_\A(\Lambda)$ be its $\A$-form.

\subsection{Young walls}
A Young wall of type $\tyX$ is a wall consisting of blocks with unit
width, unit thickness and
\begin{align*}
\xy
(0,1.5)*{};(6,1.5)*{} **\dir{-};
(0,-1.5)*{};(6,-1.5)*{} **\dir{-};
(0,-1.5)*{};(0,1.5)*{} **\dir{-};
(6,-1.5)*{};(6,1.5)*{} **\dir{-};
(25,0)*{ : \text{ half-unit height,}};
(50,4.5)*{};(56,4.5)*{} **\dir{-};
(50,-1.5)*{};(56,-1.5)*{} **\dir{-};
(50,4.5)*{};(50,-1.5)*{} **\dir{-};
(56,4.5)*{};(56,-1.5)*{} **\dir{-};
(72,0)*{ : \text{ unit height.}};
\endxy
\end{align*}
The walls are built on the ground-state wall
${}^{ \  }\xy (0,0.2)*++{\groundwall} \endxy $
by the following rules:
\begin{enumerate}
\item Blocks should be built in the pattern given below.
\item There should be no free space to the right of any block except the rightmost column.
\end{enumerate}
$$
\xy
(0,-0.5)*{};(33,-0.5)*{} **\dir{.};
(0,-1)*{};(33,-1)*{} **\dir{.};
(0,-1.5)*{};(33,-1.5)*{} **\dir{.};
(0,-2)*{};(33,-2)*{} **\dir{.};
(0,-2.5)*{};(33,-2.5)*{} **\dir{.};
(0,-3)*{};(33,-3)*{} **\dir{-};
(0,0)*{};(33,0)*{} **\dir{-};
(0,3)*{};(33,3)*{} **\dir{-};
(0,9)*{};(33,9)*{} **\dir{-};
(0,19)*{};(33,19)*{} **\dir{-};
(0,25)*{};(33,25)*{} **\dir{-};
(0,35)*{};(33,35)*{} **\dir{-};
(0,41)*{};(33,41)*{} **\dir{-};
(0,44)*{};(33,44)*{} **\dir{-};
(0,47)*{};(33,47)*{} **\dir{-};
(0,53)*{};(33,53)*{} **\dir{-};
(0,59)*{};(33,59)*{} **\dir{-};
(33,-3)*{};(33,64)*{} **\dir{-};
(27,-3)*{};(27,64)*{} **\dir{-};
(21,-3)*{};(21,64)*{} **\dir{-};
(15,-3)*{};(15,64)*{} **\dir{-};
(9,-3)*{};(9,64)*{} **\dir{-};
(3,-3)*{};(3,64)*{} **\dir{-};
(30.2,-1.5)*{_0}; (24.2,-1.5)*{_0}; (18.2,-1.5)*{_0};(12.2,-1.5)*{_0}; (6.2,-1.5)*{_0};
(30.2,1.5)*{_0}; (30.2,6)*{_1}; (30.2,15)*{\vdots}; (30.2,22)*{_\ell}; (30.2,31)*{\vdots};
(30.2,38)*{_1}; (30.2,42.5)*{_0}; (30.2,45.5)*{_0}; (30.2,50)*{_1}; (30.2,56)*{_2};
(24.2,1.5)*{_0}; (24.2,6)*{_1}; (24.2,15)*{\vdots}; (24.2,22)*{_\ell}; (24.2,31)*{\vdots};
(24.2,38)*{_1}; (24.2,42.5)*{_0}; (24.2,45.5)*{_0}; (24.2,50)*{_1}; (24.2,56)*{_2};
(18.2,1.5)*{_0}; (18.2,6)*{_1}; (18.2,15)*{\vdots}; (18.2,22)*{_\ell}; (18.2,31)*{\vdots};
(18.2,38)*{_1}; (18.2,42.5)*{_0}; (18.2,45.5)*{_0}; (18.2,50)*{_1}; (18.2,56)*{_2};
(12.2,1.5)*{_0}; (12.2,6)*{_1}; (12.2,15)*{\vdots}; (12.2,22)*{_\ell}; (12.2,31)*{\vdots};
(12.2,38)*{_1}; (12.2,42.5)*{_0}; (12.2,45.5)*{_0}; (12.2,50)*{_1}; (12.2,56)*{_2};
(6.2,1.5)*{_0}; (6.2,6)*{_1}; (6.2,15)*{\vdots}; (6.2,22)*{_\ell}; (6.2,31)*{\vdots};
(6.2,38)*{_1}; (6.2,42.5)*{_0}; (6.2,45.5)*{_0}; (6.2,50)*{_1}; (6.2,56)*{_2};
(19,-9)*{\tyX = A_{2\ell}^{(2)}};
(60,-0.5)*{};(93,-0.5)*{} **\dir{.};
(60,-1)*{};(93,-1)*{} **\dir{.};
(60,-1.5)*{};(93,-1.5)*{} **\dir{.};
(60,-2)*{};(93,-2)*{} **\dir{.};
(60,-2.5)*{};(93,-2.5)*{} **\dir{.};
(60,-3)*{};(93,-3)*{} **\dir{-};
(60,0)*{};(93,0)*{} **\dir{-};
(60,3)*{};(93,3)*{} **\dir{-};
(60,9)*{};(93,9)*{} **\dir{-};
(60,19)*{};(93,19)*{} **\dir{-};
(60,25)*{};(93,25)*{} **\dir{-};
(60,28)*{};(93,28)*{} **\dir{-};
(60,31)*{};(93,31)*{} **\dir{-};
(60,37)*{};(93,37)*{} **\dir{-};
(60,47)*{};(93,47)*{} **\dir{-};
(60,53)*{};(93,53)*{} **\dir{-};
(60,56)*{};(93,56)*{} **\dir{-};
(60,59)*{};(93,59)*{} **\dir{-};
(60,65)*{};(93,65)*{} **\dir{-};
(93,-3)*{};(93,66)*{} **\dir{-};
(87,-3)*{};(87,66)*{} **\dir{-};
(81,-3)*{};(81,66)*{} **\dir{-};
(75,-3)*{};(75,66)*{} **\dir{-};
(69,-3)*{};(69,66)*{} **\dir{-};
(63,-3)*{};(63,66)*{} **\dir{-};
(90.2,-1.5)*{_0}; (84.2,-1.5)*{_0}; (78.2,-1.5)*{_0};(72.2,-1.5)*{_0}; (66.2,-1.5)*{_0};
(90.2,1.5)*{_0}; (90.2,6)*{_1}; (90.2,15)*{\vdots}; (90.2,22)*{_{\ell-1}}; (90.2,26.5)*{_{\ell}}; (90.2,29.5)*{_{\ell}}; (90.2,34)*{_{\ell-1}}; (90.2,43)*{\vdots};
(90.2,50)*{_1}; (90.2,54.5)*{_0}; (90.2,57.5)*{_0}; (90.2,62)*{_1};
(84.2,1.5)*{_0}; (84.2,6)*{_1}; (84.2,15)*{\vdots}; (84.2,22)*{_{\ell-1}}; (84.2,26.5)*{_{\ell}}; (84.2,29.5)*{_{\ell}}; (84.2,34)*{_{\ell-1}}; (84.2,43)*{\vdots};
(84.2,50)*{_1}; (84.2,54.5)*{_0}; (84.2,57.5)*{_0}; (84.2,62)*{_1};
(78.2,1.5)*{_0}; (78.2,6)*{_1}; (78.2,15)*{\vdots}; (78.2,22)*{_{\ell-1}}; (78.2,26.5)*{_{\ell}}; (78.2,29.5)*{_{\ell}}; (78.2,34)*{_{\ell-1}}; (78.2,43)*{\vdots};
(78.2,50)*{_1}; (78.2,54.5)*{_0}; (78.2,57.5)*{_0}; (78.2,62)*{_1};
(72.2,1.5)*{_0}; (72.2,6)*{_1}; (72.2,15)*{\vdots}; (72.2,22)*{_{\ell-1}}; (72.2,26.5)*{_{\ell}}; (72.2,29.5)*{_{\ell}}; (72.2,34)*{_{\ell-1}}; (72.2,43)*{\vdots};
(72.2,50)*{_1}; (72.2,54.5)*{_0}; (72.2,57.5)*{_0}; (72.2,62)*{_1};
(66.2,1.5)*{_0}; (66.2,6)*{_1}; (66.2,15)*{\vdots}; (66.2,22)*{_{\ell-1}}; (66.2,26.5)*{_{\ell}}; (66.2,29.5)*{_{\ell}}; (66.2,34)*{_{\ell-1}}; (66.2,43)*{\vdots};
(66.2,50)*{_1}; (66.2,54.5)*{_0}; (66.2,57.5)*{_0}; (66.2,62)*{_1};
(79,-9)*{\tyX = D_{\ell+1}^{(2)}};
\endxy
$$
For example, when $\tyX =  A_{4}^{(2)}$, the following are Young walls.

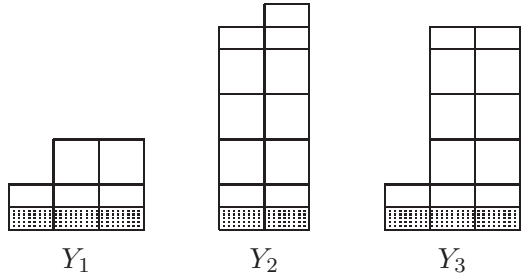
\begin{equation} \label{Ex: Young wall1}
\begin{aligned}
\xy
(22,0)*{};(40,0)*{} **\dir{-};   
(22,3)*{};(40,3)*{} **\dir{-};
(22,6)*{};(40,6)*{} **\dir{-};
(28,12)*{};(40,12)*{} **\dir{-};
(22,0)*{};(22,6)*{} **\dir{-};
(28,0)*{};(28,12)*{} **\dir{-};
(34,0)*{};(34,12)*{} **\dir{-};
(40,0)*{};(40,12)*{} **\dir{-};
(22,0.5)*{};(40,0.5)*{} **\dir{.};
(22,1)*{};(40,1)*{} **\dir{.};
(22,1.5)*{};(40,1.5)*{} **\dir{.};
(22,2)*{};(40,2)*{} **\dir{.};
(22,2.5)*{};(40,2.5)*{} **\dir{.};
(31,-4)*{Y_1};
(50,0)*{};(62,0)*{} **\dir{-};   
(50,3)*{};(62,3)*{} **\dir{-};
(50,6)*{};(62,6)*{} **\dir{-};
(50,12)*{};(62,12)*{} **\dir{-};
(50,18)*{};(62,18)*{} **\dir{-};
(50,24)*{};(62,24)*{} **\dir{-};
(50,27)*{};(62,27)*{} **\dir{-};
(56,30)*{};(62,30)*{} **\dir{-};
(50,0)*{};(50,27)*{} **\dir{-};
(56,0)*{};(56,30)*{} **\dir{-};
(62,0)*{};(62,30)*{} **\dir{-};
(50,0.5)*{};(62,0.5)*{} **\dir{.};
(50,1)*{};(62,1)*{} **\dir{.};
(50,1.5)*{};(62,1.5)*{} **\dir{.};
(50,2)*{};(62,2)*{} **\dir{.};
(50,2.5)*{};(62,2.5)*{} **\dir{.};
(56,-4)*{Y_2};
(72,0)*{};(90,0)*{} **\dir{-};   
(72,3)*{};(90,3)*{} **\dir{-};
(72,6)*{};(90,6)*{} **\dir{-};
(78,12)*{};(90,12)*{} **\dir{-};
(78,18)*{};(90,18)*{} **\dir{-};
(78,24)*{};(90,24)*{} **\dir{-};
(78,27)*{};(90,27)*{} **\dir{-};
(72,0)*{};(72,6)*{} **\dir{-};
(78,0)*{};(78,27)*{} **\dir{-};
(84,0)*{};(84,27)*{} **\dir{-};
(90,0)*{};(90,27)*{} **\dir{-};
(72,0.5)*{};(90,0.5)*{} **\dir{.};
(72,1)*{};(90,1)*{} **\dir{.};
(72,1.5)*{};(90,1.5)*{} **\dir{.};
(72,2)*{};(90,2)*{} **\dir{.};
(72,2.5)*{};(90,2.5)*{} **\dir{.};
(81,-4)*{Y_3};
\endxy
\end{aligned}
\end{equation}
For a Young wall $Y$, we write $Y = (y_k)_{k=0}^\infty = (\ldots,
y_2, y_1, y_0) $ as a sequence of its columns. Let $|y_k|$ be the
number of blocks in $y_k$ for $k\in \Z_{\ge 0}$. We denote by
$\ell(Y)$ the number of nonzero columns of $Y$, i.e. $ \ell(Y) = \#
\{  y_k \mid |y_k| \neq 0, k\ge 0 \} $, and set
\begin{align*}
|Y| = \sum_{k \in \Z_{\ge 0 }} |y_k|.
\end{align*}
Note that $|y_k|$ are weakly decreasing from right to left.
The numbers appeared in the above pattern belong to the index set $I$ and they are called {\it residues} or {\it colors} of Young walls.
When a block $b$ in a Young wall $Y$, simply $b \in Y$, has residue $i \in I$, we say that $b$ is an {\it $i$-block} in $Y$.

\begin{defn}
\begin{enumerate}
\item A column of a Young wall is called a {\it full column} if its height is a multiple of the unit length.
\item A Young wall is said to be {\it proper} if none of the full columns have the same heights.
\item An $i$-block of a proper Young wall $Y$ is called a {\it removable} $i$-block if $Y$ remains a proper Young wall after removing the block.
\item A place in a proper Young wall $Y$ is called an {\it admissible} or {\it addable} $i$-slot if $Y$ remains a proper Young wall after adding an $i$-block at the place.
\end{enumerate}
\end{defn}
\noindent
If $b$ is an $i$-block or an $i$-slot in $Y$ and
$j$ (resp.\ $k$) is the residue directly above $b$ (resp.\ directly
below $b$), then we say
\begin{align}
\res(b) = i, \qquad \res_+(b) = j, \qquad \res_-(b) = k.
\end{align}

For $i\in I$, we set
$$ a_i = \left\{
           \begin{array}{ll}
             1 & \hbox{ if $i = \ell$ and $\tyX = A_{2\ell}^{(2)}$,}  \\
             2 & \hbox{otherwise.}
           \end{array}
         \right.
 $$
Note that the weight $\sum_{i\in I} a_i \alpha_i$ (resp.\ $\frac{1}{2}\sum_{i\in I} a_i \alpha_i$) is the null root $\delta$ of $U_q(A_{2\ell}^{(2)})$ (resp.\ $U_q(D_{\ell+1}^{(2)})$).
\begin{defn}
\begin{enumerate}
\item The part of a column in a proper Young wall is called a {\it $\delta$-column} if it contains $a_0$-many $0$-blocks,  $a_1$-many $1$-blocks, \ldots,
$a_\ell$-many $\ell$-blocks  in some cyclic order.
\item A $\delta$-column in a proper Young wall $Y$ is {\it removable} if one can remove the $\delta$-column from $Y$ and the resulting is a proper Young wall.
\item A proper Young wall is said to be {\it reduced} if it has no removable $\delta$-column.
\end{enumerate}
\end{defn}
\noindent
For example, considering the Young walls in $\eqref{Ex: Young wall1}$, $Y_1$ is not proper, $Y_2$ is proper but not reduced, and $Y_3$ is proper and reduced.

Let $\mathcal{Y}(\Lambda_0)$ be the set of all reduced proper Young walls. Kashiwara operators $\tilde{e}_i$ and $\tilde{f}_i$ can be defined by considering the combinatorics of Young walls,
which yield that $\mathcal{Y}(\Lambda_0)$ is a $U_q(\tyX)$-crystal \cite{HK02,Kang03}.
\begin{Thm} \cite[Thm.7.1]{Kang03}
The crystal $\mathcal{Y}(\Lambda_0)$ is isomorphic to the crystal of
the irreducible highest weight $U_q(\tyX)$-module $V(\Lambda_0)$
with highest weight $\Lambda_0$.
\end{Thm}

\subsection{Fock space representations} \label{Sec: Fock}
Let $\mathcal{Z}(\Lambda_0)$ be the set of all proper Young walls and let
$$ \mathcal{F}(\Lambda_0) = \bigoplus_{Y \in \mathcal{Z}(\Lambda_0)} \Q(q)Y $$
be the $\Q(q)$-vector space generated by $\mathcal{Z}(\Lambda_0)$. We now define a $U_q(\tyX)$-module structure on $\mathcal{F}(\Lambda_0)$ by following \cite{KangKwon08}.

Let $Y = (y_k)_{k=0}^\infty \in \mathcal{Z}(\Lambda_0)$, and let
\begin{align*}
\wt(y_k) = \sum_{i\in I} t_{ki} \alpha_i, \quad  \wt(Y) = \Lambda_0
- \sum_{k\in \Z_{\ge0}} \wt(y_k)  ,
\end{align*}
where $ t_{ki} $ is the number of $i$-blocks in the $k$th column
$y_k$ of $Y$. If $b$ is a block or a slot in the $r$th row in $Y$,
then we set
\begin{equation} \label{Eq: delta}
\begin{aligned}
 \delta_-(b;Y) &= \left\{
                  \begin{array}{ll}
                    1 & \hbox{ if } \res(b) = \res_-(b),\ r > 1, \\
                    0 & \hbox{ otherwise.}
                  \end{array}
                \right. \\
\delta_+(b;Y) &= \left\{
                  \begin{array}{ll}
                    1 & \hbox{ if } \res(b) = \res_+(b),\ r > 1, \\
                    0 & \hbox{ otherwise.}
                  \end{array}
                \right.
\end{aligned}
\end{equation}
Note that the ground-state wall is the 0th row. For $k \in \Z_{\ge
0}$ and $i\in I$, we set
\begin{align} \label{Eq: di}
d_i(Y, k) = - \wt(y_k) (h_i) \qquad  \text{ for } k\in \Z_{\ge0}.
\end{align}
Note that
\begin{align} \label{Eq: value of di} 
d_i(Y, k) = \left\{
              \begin{array}{ll}
                -2 & \hbox{ if } i = \res(b) = \res_-(b),  \\
                -1 & \hbox{ if } i = \res(b) \ne \res_-(b),\ i \ne \res_+(b), \\
                2 & \hbox{ if } i = \res(s) = \res_+(s), \\
                1 & \hbox{ if } i = \res(s) \ne \res_+(s),\ i \ne \res_-(s),\\
                0 & \hbox{ otherwise,}
              \end{array}
            \right.
\end{align}
where $b$ is the block at the top of the $k$th column $y_k$ of $Y$, and $s$ is the slot on the top of the $k$th column $y_k$ of $Y$.

If $b$ is either a removable $i$-block or an admissible $i$-slot in the $p$th column $y_p$ of $Y$, we set
\begin{align*}
Y \nearrow b &= \text{ proper Young wall obtained by removing $b$ from $Y$ if $b$ is removable, } \\
Y \swarrow b &= \text{ proper Young wall obtained by adding an $i$-block at $b$ if $b$ is admissible, }
\end{align*}
and
\begin{equation} \label{Eq: Fock space1}
\begin{aligned}
& l(b; Y ) = \# \{ y_k \mid |y_k| = |y_p|   \} , \\
& L_i(b;Y) = \sum_{k > p} d_i(Y, k) + \Lambda_0(h_i) ,  \qquad \quad  \  R_i(b;Y) = \sum_{ p> k} d_i(Y, k).
\end{aligned}
\end{equation}

Then, we define the action of Chevalley generators $f_i, e_i$ on
$\mathcal{F}(\Lambda_0)$ as follows: For $Y
\in\mathcal{Z}(\Lambda_0)$,
\begin{equation} \label{Eq: action of f, e on Fock space}
\begin{aligned}
f_i Y &= \sum_{b} q_i^{L_i(b;Y)} \left( \frac{ 1-(-q^2)^{l(b;Y)} }{q} \right)^{\delta_-(b;Y)} (Y \swarrow b) , \\
e_i Y &= \sum_{b'} q_i^{-R_i(b';Y)} \left( \frac{ 1-(-q^2)^{l(b';Y)}
}{q} \right)^{\delta_+(b';Y)} (Y \nearrow b') ,
\end{aligned}
\end{equation}
where $b$ and $b'$ run over all admissible $i$-slots and all
removable $i$-blocks in $Y$ respectively. Note that the above
description of the Chevalley generators actions is simpler than that
in \cite[Section 5]{KangKwon08} since we only focus on the case of type $\tyX =
A_{2\ell}^{(2)}, D_{\ell+1}^{(2)}$.


\begin{Thm}\cite[Thm.5.5]{KangKwon08} \label{Thm: Fock space}
The Fock space $\mathcal{F}(\Lambda_0)$ is a $U_q(\tyX)$-module. Moreover, the submodule of $\mathcal{F}(\Lambda_0)$ generated by $\emptyset$ is isomorphic to $V(\Lambda_0)$.
\end{Thm}

\vskip 1em

\section{Quiver Hecke algebras}

We now recall the definition of the cyclotomic quiver Hecke algebra. Throughout this section, we denote by $U_q(\g)$ a quantum Kac-Moody algebra associated with a symmetrizable Cartan matrix. Let $\bR$ be a field, and, for $i,j\in I$, let
$\mathcal{Q}_{i,j}(u,v)\in\bR[u,v]$ be polynomials of the form
\begin{align*}
\mathcal{Q}_{i,j}(u,v) = \left\{
                 \begin{array}{ll}
                   \sum_{p(\alpha_i|\alpha_i)+q (\alpha_j|\alpha_j) + 2(\alpha_i|\alpha_j)=0} t_{i,j;p,q} u^pv^q & \hbox{if } i \ne j,\\
                   0 & \hbox{if } i=j,
                 \end{array}
               \right.
\end{align*}
where $(a_{ij})$ is the Cartan matrix of $U_q(\g)$, $t_{i,j;p,q} \in \bR$ are such that $t_{i,j;-a_{ij},0} \ne 0$, and $\mathcal{Q}_{i,j}(u,v) = \mathcal{Q}_{j,i}(v,u)$.
Note that the symmetric group $\mathfrak{S}_n = \langle s_k \mid k=1, \ldots, n-1 \rangle$ acts on $I^n$ by place permutations.

\begin{defn} \
The {\it cyclotomic quiver Hecke algebra} $R^{\Lambda}(n)$ associated with polynomials $(\mathcal{Q}_{i,j}(u,v))_{i,j\in I}$ and a dominant integral weight $\Lambda \in \wlP^+$
is the $\Z$-graded  $\bR$-algebra defined by three sets of generators
$$\{e(\nu) \mid \nu = (\nu_1,\ldots, \nu_n) \in I^n\}, \;\{x_k \mid 1 \le k \le n\}, \;\{\tau_l \mid 1 \le l \le n-1\} $$
subject to the following relations:
\begin{align*}
& e(\nu) e(\nu') = \delta_{\nu,\nu'} e(\nu),\ \sum_{\nu \in I^{n}} e(\nu)=1,\
x_k e(\nu) =  e(\nu) x_k, \  x_k x_l = x_l x_k,\\
& \tau_l e(\nu) = e(s_l(\nu)) \tau_l,\  \tau_k \tau_l = \tau_l \tau_k \text{ if } |k - l| > 1, 
\  \tau_k^2 e(\nu) = \mathcal{Q}_{\nu_k, \nu_{k+1}}(x_k, x_{k+1}) e(\nu), \\[5pt]
&  (\tau_k x_l - x_{s_k(l)} \tau_k ) e(\nu) =
\left\{
                                                           \begin{array}{ll}
                                                             -  e(\nu) & \hbox{if } l=k \text{ and } \nu_k = \nu_{k+1}, \\
                                                               e(\nu) & \hbox{if } l = k+1 \text{ and } \nu_k = \nu_{k+1},  \\
                                                             0 & \hbox{otherwise,}
                                                           \end{array}
                                                         \right. \\[5pt]
&( \tau_{k+1} \tau_{k} \tau_{k+1} - \tau_{k} \tau_{k+1} \tau_{k} )
e(\nu)
=\delta_{\nu_{k},\nu_{k+2}}\dfrac{\mathcal{Q}_{\nu_k,\nu_{k+1}}(x_k,x_{k+1})
- \mathcal{Q}_{\nu_k,\nu_{k+1}}(x_{k+2},x_{k+1})}{x_{k}-x_{k+2}}
e(\nu), \\
& x_1^{\langle h_{\nu_1}, \Lambda \rangle} e(\nu)=0.
\end{align*}
\end{defn}

\bigskip
The $\Z$-grading on $R^\Lambda(n)$ is given as follows:
\begin{align*}
\deg(e(\nu))=0, \quad \deg(x_k e(\nu))= ( \alpha_{\nu_k}
|\alpha_{\nu_k}), \quad  \deg(\tau_l e(\nu))=
-(\alpha_{\nu_{l}} | \alpha_{\nu_{l+1}}).
\end{align*}

For $\beta \in \mathsf{Q}^+$ with $|\beta|=n$, let
$$ I^\beta = \{ \nu=(\nu_1, \ldots, \nu_n) \in I^n \mid \alpha_{\nu_1} + \cdots + \alpha_{\nu_n} = \beta \}.$$
As $I^\beta$ is invariant under the action of $\mathfrak{S}_n$, the element $ e(\beta) = \sum_{\nu \in I^\beta} e(\nu)$ is a central idempotent of $R^\Lambda(n)$.
We set
$$R^\Lambda(\beta) = R^\Lambda(n) e(\beta).$$
In particular, when $\Lambda = \Lambda_0$, we call the algebras $\fqH(\beta)$ \emph{finite quiver Hecke algebras}.

We denote the direct sum of the Grothendieck groups of the categories $\proj(R^\Lambda (\beta))$ of finitely generated projective graded $R^\Lambda(\beta)$-modules by
$$ [\proj(R^\Lambda)] = \bigoplus_{\beta\in \rlQ^+} [\proj(R^\Lambda(\beta))]. $$
Note that $[\proj(R^\Lambda)]$ has a free $\A$-module structure
induced from the $\Z$-grading on $R^\Lambda(\beta)$, i.e. $(qM)_k =
M_{k-1}$ for a graded module $M = \bigoplus_{k \in \Z} M_k $. Let
$e(\nu, i)$ be the idempotent corresponding to the concatenation of
$\nu$ and $(i)$, and set $e(\beta, i) = \sum_{\nu \in I^\beta}
e(\nu, i)$ for $\beta \in \rlQ^+$. Define functors
\begin{align*}
E_i : \fmod(R^\Lambda(\beta + \alpha_i)) \longrightarrow \fmod(R^\Lambda(\beta)) \ \text{ and } \ 
F_i : \fmod(R^\Lambda(\beta )) \longrightarrow
\fmod(R^\Lambda(\beta+ \alpha_i)),
\end{align*}
between categories $\fmod(R^\Lambda(\beta))$ and
$\fmod(R^\Lambda(\beta + \alpha_i))$ of finitely generated graded modules
by
$$E_i(N)  = q_i^{1 - \langle h_i, \Lambda - \beta \rangle} \left( e(\beta,i)N \right) \ \text{ and } \ F_i(M) = R^\Lambda(\beta+ \alpha_i) e(\beta,i) \otimes_{R^\Lambda(\beta)}M,$$
for $M \in \fmod(R^\Lambda(\beta))$ and $N \in \fmod(R^\Lambda(\beta
+ \alpha_i))$, respectively. Then, $E_i$ and $F_i$ give a
$U_\A(\g)$-module structure to $[\proj(R^{\Lambda})]$.
\begin{Thm}[\protect{\cite[Thm.6.2]{KK11}}] \label{Thm: categorification thm}
There exists a $U_\A(\g)$-module isomorphism between $[\proj(R^{\Lambda})]$ and $V_\A(\Lambda)$.
\end{Thm}

For a graded module $ M = \bigoplus_{k\in \Z} M_k $, the {\it graded dimension} of $M$ is defined by
$$ \dim_q M = \sum_{k\in \Z} \dim(M_k)q^{k} . $$
Note that $\dim_q (q^t M) = q^t \dim_q M$. For $\Lambda \in \wlP^+$ and $\beta \in \rlQ^+$, set
$$ \df (\Lambda, \beta) = (\beta | \Lambda) - \frac{1}{2} (\beta|\beta). $$
The following proposition will be used crucially in proving the graded dimension formula.

\begin{Prop} \label{Prop: graded dim}
Let $\nu = (\nu_1, \ldots, \nu_n), \nu' = (\nu_1', \ldots, \nu_n') \in I^\beta$.
\begin{enumerate}
\item Let $v_{\Lambda}$ be the highest weight vector of the highest weight $U_q(\g)$-module $V(\Lambda)$. Then, we have
$$ e_{\nu_1} \cdots e_{\nu_n} f_{\nu_n'} \cdots f_{\nu_1'} v_{\Lambda} = q^{- \df (\Lambda, \beta) } \left(\dim_q e(\nu)R^{\Lambda}(\beta) e(\nu') \right) v_{\Lambda}. $$
\item
Let $- : \A \rightarrow \A$ be the homomorphism defined by $\overline{q} = q^{-1}$. Then
$$\overline{q^{- \df (\Lambda, \beta) } \left(\dim_q e(\nu)R^{\Lambda}(\beta) e(\nu') \right)} = q^{- \df (\Lambda, \beta) } \left(\dim_q e(\nu)R^{\Lambda}(\beta) e(\nu') \right)$$

\end{enumerate}
\end{Prop}
\begin{proof}
(1) Note that $\beta = \alpha_{\nu_1} + \cdots + \alpha_{\nu_n} = \alpha_{\nu'_1} + \cdots + \alpha_{\nu'_n}$.
For $k=1, \ldots, n$, let
$$ \mathrm{t}_k =  \frac{1}{2} (\alpha_{\nu_k}| \alpha_{\nu_k})
 - (\alpha_{\nu_k} | \Lambda-\beta + \alpha_{\nu_n} + \alpha_{\nu_{n-1}} + \cdots + \alpha_{\nu_k} ). $$
Then, we have
\begin{align*}
\mathrm{t}_1 + \cdots + \mathrm{t}_n &=  \frac{1}{2}\sum_{k=1}^n (\alpha_{\nu_k}| \alpha_{\nu_k}) - \sum_{k=1}^n
\left( (\alpha_{\nu_k} | \Lambda-\beta ) + (\alpha_{\nu_k} | \alpha_{\nu_n} + \alpha_{\nu_{n-1}} + \cdots + \alpha_{\nu_k} ) \right)\\
&= - \sum_{k=1}^n (\alpha_{\nu_k} | \Lambda-\beta ) -  \sum_{k=1}^n \left(  \frac{1}{2}(\alpha_{\nu_k}| \alpha_{\nu_k}) + (\alpha_{\nu_k} | \alpha_{\nu_n} + \alpha_{\nu_{n-1}} + \cdots + \alpha_{\nu_{k+1}} ) \right) \\
&= -(\beta | \Lambda-\beta ) -  \frac{1}{2}\sum_{k,l = 1, \ldots, n } (\alpha_{\nu_k}| \alpha_{\nu_l})\\
&= -(\beta | \Lambda ) +  \frac{1}{2} (\beta| \beta)\\
&= - \df(\Lambda, \beta) .
\end{align*}
Therefore, it follows from $ q_i^{1 - \langle h_i, \Lambda - \gamma \rangle} = q^{(\alpha_i| \alpha_i)/2 - ( \alpha_i| \Lambda - \gamma )} $ that
\begin{align*}
E_{\nu_1} \cdots E_{\nu_n} F_{\nu_n'} \cdots F_{\nu_1'} R^\Lambda(0) & =  E_{\nu_1} \cdots E_{\nu_n}  R^\Lambda(\beta) e(\nu') \\
&=  q^{\mathrm{t}_n} E_{\nu_1} \cdots E_{\nu_{n-1}} e(\beta-\alpha_{\nu_n}, \nu_n) R^\Lambda(\beta) e(\nu') \\
& \qquad \qquad \qquad \vdots \\
&=  q^{\mathrm{t}_1 + \cdots + \mathrm{t}_n} e(\nu) R^\Lambda(\beta) e(\nu')\\
&=  q^{-\df(\Lambda, \beta)} e(\nu) R^\Lambda(\beta) e(\nu'),
\end{align*}
which implies the assertion by Theorem \ref{Thm: categorification
thm}.

(2) The assertion follows immediately from (1) and the fact that
$e_{\nu_1} \cdots e_{\nu_n} f_{\nu_n'} \cdots f_{\nu_1'}
v_{\Lambda}$ is bar-invariant.
\end{proof}

\vskip 1em

\section{standard tableaux for proper Young walls}

\subsection{Standard tableaux} \label{Subsec: standard tableaux}
Let $Y \in \mathcal{Z}(\Lambda_0)$ be a proper Young wall of type $\tyX$ and let $n = |Y|$.
%
%

\begin{defn}  A {\it standard tableau} of shape $Y $ is a sequence of proper Young walls
$$ \emptyset = Y^{(0)} \subset Y^{(1)} \subset \cdots \subset Y^{(n-1)} \subset Y^{(n)} = Y, $$
where $|Y^{(k+1)}| - |Y^{(k)}| = 1$ for $k=0,1,\ldots,n-1$.
\end{defn}
\noindent
For a standard tableau $T = (Y^{(0)}, \ldots, Y^{(n)})$ of shape $Y$, let
$T_{\le k}$ (resp.\ $T_{< k}$) denote the standard tableau $(Y^{(0)}, \ldots, Y^{(k)})$ (resp.\ $(Y^{(0)}, \ldots, Y^{(k-1)})$)
consisting of the first $k+1$ (resp.\ $k$) proper Young walls of $T$. Let
$$ \ST(Y) = \text{ the set of all standard tableaux of shape $Y$. } $$

A standard tableau also can be described by labeling blocks with
numbers. For a standard tableau $T = (Y^{(0)}, \ldots, Y^{(n)})$ of
shape $Y$, if we label the block of $ Y^{(k)} \setminus Y^{(k-1)} $
with the number $k$ for all $k$, then the resulting is a proper
Young wall whose blocks are labeled with entries $1, 2, \ldots, n$.
For example, when $X = A_{4}^{(2)}$, if we have the following
standard tableau $T$
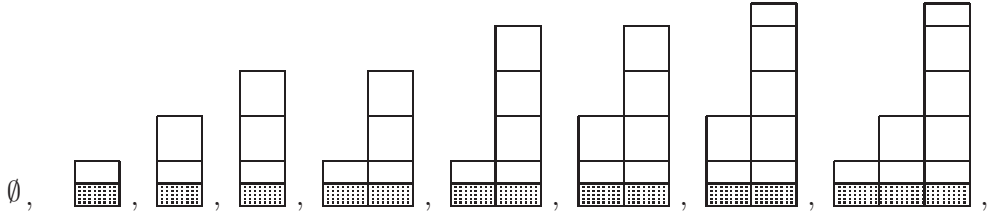
\begin{align*}
\xy
(3,-10)*{ \emptyset };
(11,-12)*{};(17,-12)*{} **\dir{-}; (11,-9)*{};(17,-9)*{} **\dir{-}; (11,-6)*{};(17,-6)*{} **\dir{-}; (11,-12)*{}; (11,-6)*{} **\dir{-}; (17,-12)*{};(17,-6)*{} **\dir{-}; 
(11,-11.5)*{};(17,-11.5)*{} **\dir{.}; (11,-11)*{};(17,-11)*{} **\dir{.};
(11,-10.5)*{};(17,-10.5)*{} **\dir{.}; (11,-10)*{};(17,-10)*{} **\dir{.}; (11,-9.5)*{};(17,-9.5)*{} **\dir{.};
(22,-12)*{};(28,-12)*{} **\dir{-}; (22,-9)*{};(28,-9)*{} **\dir{-}; (22,-6)*{};(28,-6)*{} **\dir{-}; (22,0)*{};(28,0)*{} **\dir{-}; (22,-12)*{}; (22,0)*{} **\dir{-}; (28,-12)*{};(28,0)*{} **\dir{-}; 
(22,-11.5)*{};(28,-11.5)*{} **\dir{.}; (22,-11)*{};(28,-11)*{} **\dir{.};
(22,-10.5)*{};(28,-10.5)*{} **\dir{.}; (22,-10)*{};(28,-10)*{} **\dir{.}; (22,-9.5)*{};(28,-9.5)*{} **\dir{.};
(33,-12)*{};(39,-12)*{} **\dir{-}; (33,-9)*{};(39,-9)*{} **\dir{-}; (33,-6)*{};(39,-6)*{} **\dir{-}; (33,0)*{};(39,0)*{} **\dir{-}; (33,6)*{};(39,6)*{} **\dir{-}; 
(33,-12)*{}; (33,6)*{} **\dir{-}; (39,-12)*{};(39,6)*{} **\dir{-};
(33,-11.5)*{};(39,-11.5)*{} **\dir{.}; (33,-11)*{};(39,-11)*{} **\dir{.};
(33,-10.5)*{};(39,-10.5)*{} **\dir{.}; (33,-10)*{};(39,-10)*{} **\dir{.}; (33,-9.5)*{};(39,-9.5)*{} **\dir{.};
(44,-12)*{};(56,-12)*{} **\dir{-}; (44,-9)*{};(56,-9)*{} **\dir{-}; (44,-6)*{};(56,-6)*{} **\dir{-}; (50,0)*{};(56,0)*{} **\dir{-}; (50,6)*{};(56,6)*{} **\dir{-}; 
(44,-12)*{}; (44,-6)*{} **\dir{-}; (50,-12)*{};(50,6)*{} **\dir{-}; (56,-12)*{};(56,6)*{} **\dir{-};
(44,-11.5)*{};(56,-11.5)*{} **\dir{.}; (44,-11)*{};(56,-11)*{} **\dir{.};
(44,-10.5)*{};(56,-10.5)*{} **\dir{.}; (44,-10)*{};(56,-10)*{} **\dir{.}; (44,-9.5)*{};(56,-9.5)*{} **\dir{.};
(61,-12)*{};(73,-12)*{} **\dir{-}; (61,-9)*{};(73,-9)*{} **\dir{-}; (61,-6)*{};(73,-6)*{} **\dir{-}; (67,0)*{};(73,0)*{} **\dir{-}; (67,6)*{};(73,6)*{} **\dir{-};  (67,12)*{};(73,12)*{} **\dir{-}; 
(61,-12)*{}; (61,-6)*{} **\dir{-}; (67,-12)*{};(67,12)*{} **\dir{-}; (73,-12)*{};(73,12)*{} **\dir{-};
(61,-11.5)*{};(73,-11.5)*{} **\dir{.}; (61,-11)*{};(73,-11)*{} **\dir{.};
(61,-10.5)*{};(73,-10.5)*{} **\dir{.}; (61,-10)*{};(73,-10)*{} **\dir{.}; (61,-9.5)*{};(73,-9.5)*{} **\dir{.};
(78,-12)*{};(90,-12)*{} **\dir{-}; (78,-9)*{};(90,-9)*{} **\dir{-}; (78,-6)*{};(90,-6)*{} **\dir{-}; (78,0)*{};(90,0)*{} **\dir{-}; (84,6)*{};(90,6)*{} **\dir{-};  (84,12)*{};(90,12)*{} **\dir{-}; 
(78,-12)*{}; (78,0)*{} **\dir{-}; (84,-12)*{};(84,12)*{} **\dir{-}; (90,-12)*{};(90,12)*{} **\dir{-};
(78,-11.5)*{};(90,-11.5)*{} **\dir{.}; (78,-11)*{};(90,-11)*{} **\dir{.};
(78,-10.5)*{};(90,-10.5)*{} **\dir{.}; (78,-10)*{};(90,-10)*{} **\dir{.}; (78,-9.5)*{};(90,-9.5)*{} **\dir{.};
(95,-12)*{};(107,-12)*{} **\dir{-}; (95,-9)*{};(107,-9)*{} **\dir{-}; (95,-6)*{};(107,-6)*{} **\dir{-}; (95,0)*{};(107,0)*{} **\dir{-}; (101,6)*{};(107,6)*{} **\dir{-};  (101,12)*{};(107,12)*{} **\dir{-};  
(101,15)*{};(107,15)*{} **\dir{-};
(95,-12)*{}; (95,0)*{} **\dir{-}; (101,-12)*{};(101,15)*{} **\dir{-}; (107,-12)*{};(107,15)*{} **\dir{-};
(95,-11.5)*{};(107,-11.5)*{} **\dir{.}; (95,-11)*{};(107,-11)*{} **\dir{.};
(95,-10.5)*{};(107,-10.5)*{} **\dir{.}; (95,-10)*{};(107,-10)*{} **\dir{.}; (95,-9.5)*{};(107,-9.5)*{} **\dir{.};
(112,-12)*{};(130,-12)*{} **\dir{-}; (112,-9)*{};(130,-9)*{} **\dir{-}; (112,-6)*{};(130,-6)*{} **\dir{-}; (118,0)*{};(130,0)*{} **\dir{-}; (124,6)*{};(130,6)*{} **\dir{-};  (124,12)*{};(130,12)*{} **\dir{-};  
(124,15)*{};(130,15)*{} **\dir{-};
(112,-12)*{}; (112,-6)*{} **\dir{-}; (118,-12)*{};(118,0)*{} **\dir{-}; (124,-12)*{};(124,15)*{} **\dir{-}; (130,-12)*{};(130,15)*{} **\dir{-};
(112,-11.5)*{};(130,-11.5)*{} **\dir{.}; (112,-11)*{};(130,-11)*{} **\dir{.};
(112,-10.5)*{};(130,-10.5)*{} **\dir{.}; (112,-10)*{};(130,-10)*{} **\dir{.}; (112,-9.5)*{};(130,-9.5)*{} **\dir{.};
(5,-12)*{ , }; (19,-12)*{ , }; (30,-12)*{ , }; (41,-12)*{ , }; (58,-12)*{ , }; (75,-12)*{ , };
(92,-12)*{ , }; (109,-12)*{ , }; (132,-12)*{ , };
\endxy
\end{align*}
then the expression of $T$ using numbers is give as follows:
\begin{align} \label{Ex: tableau1}
\xy
(103,0)*{ T =  };
(112,-12)*{};(130,-12)*{} **\dir{-}; (112,-9)*{};(130,-9)*{} **\dir{-}; (112,-6)*{};(130,-6)*{} **\dir{-}; (118,0)*{};(130,0)*{} **\dir{-}; (124,6)*{};(130,6)*{} **\dir{-};  (124,12)*{};(130,12)*{} **\dir{-};  
(124,15)*{};(130,15)*{} **\dir{-};
(112,-12)*{}; (112,-6)*{} **\dir{-}; (118,-12)*{};(118,0)*{} **\dir{-}; (124,-12)*{};(124,15)*{} **\dir{-}; (130,-12)*{};(130,15)*{} **\dir{-};
(112,-11.5)*{};(130,-11.5)*{} **\dir{.}; (112,-11)*{};(130,-11)*{} **\dir{.};
(112,-10.5)*{};(130,-10.5)*{} **\dir{.}; (112,-10)*{};(130,-10)*{} **\dir{.}; (112,-9.5)*{};(130,-9.5)*{} **\dir{.};
(115.3,-7.5)*{ _8 };(121.3,-7.5)*{ _4 }; (127.3,-7.5)*{ _1 };
(121.3,-3)*{ _6 }; (127.3,-3)*{ _2 };
(127.3,3)*{ _3 };
(127.3,9)*{ _5 };
(127.3,13.5)*{ _7 };
\endxy
\end{align}
We will use the above expression for standard tableaux if there is no confusion.

Just as usual standard tableaux for partitions, standard tableaux
for proper Young walls also can be defined by considering their
entries. A {\it tableau} $T$ of shape $Y$ is a labeling of blocks in
$Y$ with numbers, and let
$$\sh(T) = Y.$$
The {\it canonical tableau} $T^Y$ is the tableau in which the number $1,2, \ldots, n$ appear in order from bottom to top and from right to left.
For example, the following is a canonical tableau.
\begin{align} \label{Eq: canonical tableau}
\xy
(103,0)*{ T^Y =  };
(112,-12)*{};(130,-12)*{} **\dir{-}; (112,-9)*{};(130,-9)*{} **\dir{-}; (112,-6)*{};(130,-6)*{} **\dir{-}; (118,0)*{};(130,0)*{} **\dir{-}; (124,6)*{};(130,6)*{} **\dir{-};  (124,12)*{};(130,12)*{} **\dir{-};  
(124,15)*{};(130,15)*{} **\dir{-};
(112,-12)*{}; (112,-6)*{} **\dir{-}; (118,-12)*{};(118,0)*{} **\dir{-}; (124,-12)*{};(124,15)*{} **\dir{-}; (130,-12)*{};(130,15)*{} **\dir{-};
(112,-11.5)*{};(130,-11.5)*{} **\dir{.}; (112,-11)*{};(130,-11)*{} **\dir{.};
(112,-10.5)*{};(130,-10.5)*{} **\dir{.}; (112,-10)*{};(130,-10)*{} **\dir{.}; (112,-9.5)*{};(130,-9.5)*{} **\dir{.};
(115.3,-7.5)*{ _8 };(121.3,-7.5)*{ _6 }; (127.3,-7.5)*{ _1 };
(121.3,-3)*{ _7 }; (127.3,-3)*{ _2 };
(127.3,3)*{ _3 };
(127.3,9)*{ _4 };
(127.3,13.5)*{ _5 };
\endxy
\end{align}
Note that canonical tableaux are standard.
For a tableau $T$, $i\in \Z_{\ge 0}$, $j\in \Z_{> 0}$, let
 $$T(i,j)\ = \ \text{the entry of the block at $(i,j)$-position in $T$},$$
where $(i,j)$-position is the position in $\sh(T)$ at the $i$th column and the $j$th row.
If there is no block at $(i,j)$-position, we set $T(i,j) = \infty $.
For example, if $T$ is the tableau of $\eqref{Ex: tableau1}$, then $T(0,1) = 1$, $T(0,4) = 5$, and $T(2,1) = 8$.
The following is another description for standard tableaux. Since the proof is straightforward, we leave it to the reader.

\begin{Prop} \label{Prop: standard}
Let $Y \in \mathcal{Z}(\Lambda_0)$ be a proper Young wall and let $T$ be a tableau of shape $Y$. $T$ is standard if and only if, for all possible $i \in \Z_{\ge0},\ j\in \Z_{>0}$,
\begin{align*} 
T(i,j) < T(i,j+1),\quad T(i,j) < T(i+1,j-\delta(i,j)) ,
\end{align*}
where, if $b$ is the block of $Y$ at $(i,j)$-position, then
$$\delta(i,j) = \left\{
                     \begin{array}{ll}
                       0 & \hbox{ if } \res(b) = \res_-(b), \\
                       1 & \hbox{ otherwise.}
                     \end{array}
                   \right.
$$
\end{Prop}

For a proper Young wall $Y = (y_k)_{k=0}^\infty \in \mathcal{Z}(\Lambda_0)$, the {\it associated partition}
$\lambda_Y = ( |y_0|,|y_1|,|y_2|,\ldots )$ of $Y$  is an $h$-strict partition,
where
\begin{align} \label{Eq: h strict}
h  = \left\{
         \begin{array}{ll}
           2\ell+1 & \hbox{ if } \tyX = A_{2\ell}^{(2)}, \\
           \ell+1 & \hbox{ if } \tyX = D_{\ell+1}^{(2)}.
         \end{array}
       \right.
\end{align}
In the cases of $h=1$, $\infty$,
Proposition \ref{Prop: standard} may give a definition for usual standard tableaux of Young diagrams and shifted Young diagrams.
The case $h=1$ may be viewed as the case $\ell = 0$. In that case, all blocks have half-unit height and all residue of Young walls are $0$, so $\delta(i,j) = 0$ for all $i,j$.
Thus, entries in rows and columns of a standard tableau of $Y$ increase from right to left and from bottom to top, respectively.
The case $h=\infty$ may be regarded as the case $\ell = \infty$. In that case, all residue in each column are different, so $\delta(i,j) = 1$ for all $i,j$.
If, for each $k$, we move up the $k$-column of a standard tableau of shape $Y$ by $k$, then the resulting
 tableau is a tableau of a shifted Young diagram. Moreover, entries in rows and columns of the resulting tableau increase from right and left and from bottom to top, respectively.

Let $Y$ be a proper Young wall. If the associated partition $\lambda_Y$ of $Y$ is strict, then
we define $\ST_\infty(Y)$ to be the set of all standard tableaux of shape $Y$ such that
\begin{align} \label{Eq: def of STinf}
T(i,j) < T(i,j+1),\quad T(i,j) < T(i+1,j-1)  \quad \text{ for all possible }i,j.
\end{align}
We set $\ST_\infty(Y) =\emptyset$ if the associated partition of $Y$ is not strict.
By definition, we have
\begin{align*}
\ST_\infty(Y) \subset \ST(Y).
\end{align*}
Moreover, if the associated partition $\lambda_Y$ is strict, then
the set $\ST_\infty(Y)$ is bijective to the set of all standard tableaux of shape $\lambda_Y$ via the correspondence for $h=\infty$ in the above paragraph.



 \subsection{Weak order on $\ST(Y)$}
Let $\sg_n$ be the symmetric group of permutations of $\{1,2, \ldots, n \}$, which is generated by simple transposition $s_i := (i,i+1)$ for $i=1,\ldots, n-1$.
For $w\in \sg_n$, let $l(w)$ be the {\it length} of $w$, i.e., the number of simple transpositions in its reduced expression.
As a permutation $w \in \sg_n$ may be regarded as
the word $ (w(1) w(2) \cdots w(n) )$, we think of $\sg_n$ as the set of words with $n$ distinct letters $1,2, \ldots, n$.
Note that, for a permutation $w = (w_1 \cdots w_i w_{i+1} \cdots w_n)$,
\begin{align}
ws_i &= (w_1 \cdots w_{i+1} w_i \cdots w_n), \label{Eq: going up in weak order 1} \\
l(w s_i) &= l(w) + 1 \quad \text{if and only if} \quad  w_i <
w_{i+1}. \label{Eq: going up in weak order}
\end{align}
For $u,w \in \sg_n$, we say $u < w$ in {\it weak order} if there
exist $s_{i_1}, s_{i_2}, \ldots, s_{i_t} $ such that
$$l(s_{i_k}\cdots s_{i_2}s_{i_1} u ) = l(u) + k  \ \ (1 \le k \le t),  \quad \quad w = s_{i_t}\cdots s_{i_1} u.  $$
For a subset $V \subset \sg_n$, the set
$$ \sg_n/V = \{ w\in \sg_n \mid l(wv) = l(w) + l(v) \text{ for all } v\in V  \} $$
is called a {\it generalized quotient} for $\sg_n$.
\begin{Thm} \cite[Thm.4.1]{BW88} \label{Thm: BW88}
Let $\sg_n/V$ be a generalized quotient.
\begin{enumerate}
\item There exists a permutation $v\in \sg_n$ such that
$$\sg_n/V = \{ w\in \sg_n \mid \mathrm{id} \le w \le v \} .$$
\item $\sg_n/V$ is a lattice under the weak order.
\end{enumerate}
\end{Thm}

We now define a partial order on $\ST(Y)$ induced from the weak order on $\sg_n$.
Let $Y$ be a proper Young wall with $n$ blocks.
For a tableau $T$ of shape $Y$, let $w(T)$ be the word obtained by reading entries of $T$ column by column from bottom to top and from right to left.
Since, for $T\in \ST(Y)$, $w(T)$ is a word of $n$ distinct letters $1,2, \ldots, n$,  it can be viewed as a permutation of $\sg_n$. Thus, the set
$$ w(\ST(Y)) = \{ w(T) \mid T \in \ST(Y)  \} $$
can be regarded as a subset of $\sg_n$.
For example, $w(T^Y) = (12\cdots n)$ and, if $T$ is the tableau $\eqref{Ex: tableau1}$, $w(T)$ is given as
$$ w(T) = (12357468) \in \sg_n. $$

For $T, T' \in \ST(Y)$, we define
\begin{align} \label{Eq: weak order}
T < T'
\end{align}
if $w(T) < w(T')$ in the weak order of $\sg_n$, which gives
a partial order on $\ST(Y)$. Note that the canonical tableau $T^Y$
is a unique smallest tableau in $\ST(Y)$. The following theorem is a
generalization of \cite[Thm.7.2]{BW88} to standard tableaux for
Young walls.

\begin{Thm} \label{Thm: weak order}
Let $Y$ be a proper Young wall with $n$ blocks.
\begin{enumerate}
\item The subset $w(\ST(Y))$ is a generalized quotient for $\sg_n$.
\item The poset $\ST(Y)$ is a graded poset with a unique maximum element.
\item The poset $\ST(Y)$ is a lattice.
\item For $T, T' \in \ST(Y)$ with $T<T'$, there exist simple transpositions $ s_{i_1}, s_{i_2}, \ldots, s_{i_t}$ such that,
$$ T' = s_{i_t} \cdots s_{i_1} T, \quad s_{i_j} \cdots s_{i_1} T \in \ST(Y) \text{ for all $1\le j \le t$}, $$
where $w T$ is the tableau obtained from $T$ by replacing entry $k$ with $w(k)$ for $k=1,\ldots,n$.
\end{enumerate}
\end{Thm}
\begin{proof}
For $i\in \Z_{\ge 0 },\ j \in \Z_{>0}$, if $Y$ has a block at $(i,j)$-position,
let $t_{ij}$ be the entry of the top block of the canonical tableau $T^Y$ in the $i$th column and
$$ o_{ij} = T^Y(i,j), \quad  l_{ij} = T^Y(i+1, j-\delta(i,j) ),$$
where $\delta(i,j)$ is given in Proposition \ref{Prop: standard}.
Note that $ o_{ij} \le t_{ij} \le l_{ij} $. For each $i,j$ with $l_{ij} \ne \infty$, let
\begin{align*}
v_{ij} = &(o_{ij} ,o_{ij} +1) (o_{ij}+1,o_{ij}+2) \cdots (t_{ij}-1,t_{ij}) \\
& \qquad (l_{ij}-1, l_{ij}) (l_{ij}-2, l_{ij}-1) \cdots (t_{ij}+1, t_{ij}+2) (t_{ij}, t_{ij}+1).
\end{align*}
Note that, for a tableau $T$, $w(T)v_{ij}$ is obtained from $w(T)$ by
first moving $ T(i,j)$ to the top block of the $i$th column, and
then moving $T(i+1, j-\delta(i,j))$ to the bottom block of $(i+1)$th
column, and lastly interchanging these two entries.
We now define
$$V = J \cup \{ v_{ij} \mid l_{ij} \ne \infty \}\subset \sg_n ,$$
where $J = \{ (k,k+1) \mid k = 1,2, \ldots, n-1,\ k \ne t_{ij} \ \  \text{ for all possible }i,j \}  $.
Let $T$ be a tableau of shape $Y$. By Proposition \ref{Prop: standard}$, \eqref{Eq: going up in
weak order 1}$ and $\eqref{Eq: going up in weak order}$, we have
\begin{align*}
T \in \ST(Y) \ \  \Leftrightarrow \ \
\left\{
  \begin{array}{ll}
    l(w(T)u) = l(w(T)) + l(u) & \text{ for all $u\in J$}, \\
    l(w(T)v_{ij}) =  l(w(T)) + l(v_{ij}) & \text{ for all $i, j$.}
  \end{array}
\right.
\end{align*}
Thus, $ w(\ST(Y)) = \{ w \in \sg_n \mid l(w v) = l(w) + l(v) \ \text{ for all } v\in V  \} ,$
which is a generalized quotient.

The remaining assertions follow from (1) and Theorem \ref{Thm: BW88}.
\end{proof}

\vskip 1em

\section{Laurent polynomials associated with proper Young walls  } \label{Sec: Laurent poly}




Let $Y \in \mathcal{Z}(\Lambda_0)$ be a proper Young wall with $n$
blocks. For a standard tableau
$T \in \ST(Y)$ of shape $Y$, we set $ \dg(\emptyset) = 0 $ and $
\cdg(\emptyset) = 0 $, and, if $b$ denotes the  block of $Y$
with entry $n$ and $i = \res(b)$, then we define inductively
\begin{equation} \label{Eq: degree codegree}
\begin{aligned}
\dg(T) = \mathsf{d}_i \cdot L_i(b;Y) + \dg(T_{<n}), \quad \quad \cdg(T) =
\mathsf{d}_i \cdot R_i(b;Y) + \cdg(T_{<n}),
\end{aligned}
\end{equation}
where $\mathsf{d}_i$ are the integers defined by $\eqref{Eq: diagonal matrix}$.
In a similar manner, we set $ \qE_q(\emptyset) = 1 $ and $
\qF_q(\emptyset) = 1 $, and, if $b$ denotes the $i$-block
of $Y$ with entry $n$, then we define inductively
\begin{equation} \label{Eq: Laurent poly}
\begin{aligned}
\qF_q(T) &= q_i^{L_i(b;Y^-)} \left( \frac{ 1-(-q^2)^{l(b; Y^-)} }{q} \right)^{ \delta_-(b;Y^-) } \qF_q(T_{<n}), \\
\qE_q(T) &= q_i^{-R_i(b;Y)} \left( \frac{ 1-(-q^2)^{l(b; Y )} }{q}
\right)^{ \delta_+(b;Y) } \qE_q(T_{<n}),
\end{aligned}
\end{equation}
where $Y^- = Y \nearrow b$, i.e. $Y^- = \sh(T_{<n})$, $ L_i(b;Y^-)
$, $R_i(b;Y)$, $\delta_-(b;Y)$, $\delta_+(b;Y)$, and $l(b; Y )$ are
given as $\eqref{Eq: delta}$ and $\eqref{Eq: Fock space1}$. Since
$L_i(b;Y)=L_i(b;Y^-)$, one can easily check that
$$ \qF_q(T) = q^{\dg(T)} \qF'_q(T), \qquad \qE_q(T)= q^{-\cdg(T)} \qE'_q(T) $$
where $\qF'_q(T)$ and $\qE'_q(T)$ are defined inductively as
follows:
$$
\qF'_q(T) = \left( \frac{ 1-(-q^2)^{l(b; Y^-)} }{q}
\right)^{ \delta_- (b;Y^-)} \qF'_q(T_{<n}), \ \ \qE'_q(T) = \left(
\frac{ 1-(-q^2)^{l(b; Y )} }{q} \right)^{ \delta_+  (b;Y)}
\qE'_q(T_{<n}).
$$
Here $\qF'_q(\emptyset)=1$ and $\qE'_q(\emptyset)=1$.


Recall the number $h$ defined as $\eqref{Eq: h strict}$. We set
\begin{align} \label{Eq: hst}
\hst_Y = ( \hst_1, \hst_2, \ldots ), \qquad |\hst_Y| = \sum_{j=1}^{\infty} \hst_j,
\end{align}
where $\hst_j$ is the number of columns $y_k$ of $Y$ whose heights are $hj$, i.e. $\varpi_j = \# \{ y_k \mid |y_k| = hj \}$.
Note that, if $a$ is the block at the top of the $k$th column $y_k$, then
\begin{align} \label{Eq: h1}
 h \mid |y_k| \quad \text{ if and only if } \quad \res(a) = \res_+(a).
\end{align}

The following lemma can be understood as an analogue of
\cite[Lemma 3.12]{BKW11}.

\begin{Lem} \label{Lem: def+}
Let $Y$ be a proper Young wall with $n$ blocks and let $ \beta = \Lambda_0 - \wt(Y)$. Then,
for $T \in \ST(Y)$, we have
$$ \dg(T)+\cdg(T)=\df(\Lambda_0, \beta)- |\hst_Y|.$$
\end{Lem}
\begin{proof}
Let $b$ be the block in $T$ with entry $n$ in the $k$th column and
$$i = \res(b), \qquad Y^- = \sh(T_{<n}).$$
Since $ \mathsf{d}_i = 1$ if either $i=0$ ($\tyX = A_{2\ell}^{(2)}$) or $i=0,\ell$ ($\tyX = D_{\ell+1}^{(2)}$), it follows from $\eqref{Eq: value of di}$ that
$$ \mathsf{d}_i (1+ d_i(Y;k) ) = \left\{
                   \begin{array}{ll}
                     -1 & \hbox{ if } i = \res(b) = \res_-(b), \\
                     1 & \hbox{ if } i = \res(b) = \res_+(b), \\
                     0 & \hbox{ otherwise,}
                   \end{array}
                 \right.
 $$
which yields, by $\eqref{Eq: h1}$,
$$ |\hst_Y| = \mathsf{d}_i (1+ d_i(Y;k) ) + |\hst_{Y^-}| . $$

We use induction on $n$. If $n = 0$, then there is noting to prove. Suppose that the assertion holds for all proper Young walls $Y'$ with $|Y'| < n$.
Then, since
\begin{align*}
\df(\Lambda_0, \beta - \alpha_i) &= \df(\Lambda_0, \beta)  -(\Lambda_0 - \beta| \alpha_i) - (\alpha_i | \alpha_i)/2, \\
&= \df(\Lambda_0, \beta) - \mathsf{d}_i (\Lambda_0 - \beta)(h_i) - \mathsf{d}_i ,
\end{align*}
we have
\begin{align*}
\dg(T)+\cdg(T) &= \mathsf{d}_i (L_i(b;Y) + R_i(b;Y)) + \dg(T_{<n})+\cdg(T_{<n})\\
&= \mathsf{d}_i \left( (\Lambda_0-\beta)(h_i) -d_i(Y;k) \right) + \df(\Lambda_0, \beta-\alpha_i)- |\hst_{Y^-}| \\
&=  \df(\Lambda_0, \beta)  -  \mathsf{d}_i ( 1+ d_i(Y;k)  ) - |\hst_{Y^-}| \\
&=  \df(\Lambda_0, \beta)  -  |\hst_{Y}|,
\end{align*}
which complete the proof.
\end{proof}


\begin{defn} \label{Def:ltb}
\begin{enumerate}
\item For a proper Young wall $Y$, let
\begin{equation} \label{def:OCCD}
\begin{aligned}
C_-(Y) &= \{ b \in Y \mid \delta_-(b;Y) =1 \}, \quad \  \  \ c_-(Y)=|C_-(Y)|,\\
C_+(Y) &= \{ b \in Y \mid \delta_+(b;Y) =1 \},\ \quad  \ \ c_+(Y)=|C_+(Y)|.
\end{aligned}
\end{equation}
Note that $C_-(Y) \cap C_+(Y) = \emptyset$.
\item Let $b$ be a block in a proper Young wall $Y$ in the $i$th column and the $j$th row.
\begin{enumerate}
\item
If $b$ is in $C_-(Y)$, then we define
$$ l(b;T) = \max\{ k \in \Z_{\ge 0} \mid T(i+k,j-1) < T(i,j) \} $$
\item
If $b$ is in $C_+(Y)$, then we define
$$ l(b;T) = \max\{k \in \Z_{\ge 0} \mid  T(i-k,j+1) > T(i,j) \}. $$
\end{enumerate}
Note that $T(k,l) = \infty$ if there is no block at the $(k,l)$-position in $T$.
\end{enumerate}
\end{defn}

The following theorem gives us explicit descriptions for $\qF_q(T)$ and $\qE_q(T)$.

\begin{Thm} \label{Thm: qF over qE}
Let $Y = (y_k)_{k=0}^\infty$ be a proper Young wall with $n$ blocks.
\begin{enumerate}
\item For any $T \in \ST(Y)$, we have
\begin{align*}
\qF'_q(T) &= q^{-c_-(Y)} \prod_{b \in C_-(Y)}  (1- (-q^2)^{l(b;T)+1}),  \\
\qE'_q(T) &= q^{-c_+(Y)}\prod_{b \in C_+(Y)}  (1- (-q^2)^{l(b;T)+1}),
\end{align*}
\noindent
and hence
\begin{align*}
\qF_q(T) &= q^{\dg(T)-c_-(Y)} \prod_{b \in C_-(Y)}(1- (-q^2)^{l(b;T)+1}), \\
\qE_q(T) &= q^{-\cdg(T)-c_+(Y)}\prod_{b \in C_+(Y)} (1- (-q^2)^{l(b;T)+1}),
\end{align*}
where $C_-(Y)$, $C_+(Y)$, $c_-(Y)$ and $c_+(Y)$ are given as $\eqref{def:OCCD}$.
\item In particular, we have
\begin{align*}
\qF_q(T^Y) &= q^{o(Y)-c_-(Y)} (1+q^2 )^{c_-(Y)}, \\
\qE_q(T^Y) &=  q^{d(Y)-c_+(Y)} \prod_{b\in C_+(Y) }  (1-(-q^2)^{l_R(b;Y)}) ,
\end{align*}
where $o(Y) = (\Lambda_0 | \Lambda_0 - \wt(Y))$, $d(Y) = \sum_{  1 \le k < l \le \ell(Y)} (\wt( y_l )\mid \wt( y_k  ) )$, and, if $b$ is in the $p$th column, then
\begin{align} \label{eq: lRby}
l_R(b;Y) = \left\{
               \begin{array}{ll}
                 \# \{ y_k \mid p \ge k,\  |y_k|=|y_p| \} & \hbox{ if } \text{ $b$ is at the top of its column}, \\
                 1 & \hbox{ otherwise.}
               \end{array}
             \right.
\end{align}
\end{enumerate}
\end{Thm}
\begin{proof}
(1) It can be proved easily from definitions by induction on $n=|Y|$.
%
%
%

(2) By the definition of the canonical tableaux $T^Y$, it follows from $\eqref{Eq: di}$ and $\eqref{Eq: Fock space1}$ that
$$ \dg(T^Y) = o(Y), \qquad  \cdg(T^Y) = - d(Y). $$

The configuration of entries of $T^Y$ gives
$$ l(b;T^Y)=0 \text{ if } b\in C_-(b), \qquad
l(b;T^Y)= l_R(b;Y) - 1 \text{ if } b\in C_+(b),  $$
which implies the assertion.
\end{proof}

Theorem \ref{Thm: weak order} tells us that, for any standard tableau $T \in \ST(Y)$, there exist a sequence $\{T_k\}_{k=0}^{t}$ of standard tableaux such that
\begin{align*}
& T_0 = T^Y, \quad T_t = T, \\
& T_k = s\ T_{k-1} \text{ with } T_k > T_{k-1} \text{ for some simple transposition $s \in \sg_n$}.
\end{align*}
The following proposition with Theorem \ref{Thm: weak order} allows to compute the Laurent polynomials $\qF_q(T)$ and $\qE_q(T)$ inductively.

\begin{Prop} \label{Prop: ratio for E, F}
Let $T, T' \in \ST(Y)$ be standard tableaux of shape $Y$ such that $ T < T' $ and $s_k T = T'$ for some simple transposition $s_k \in \sg_n$.
We denote by $b_1$ (resp.\ $b_2$) the block with entry $k$ (resp.\ $k+1$) in $T$, and set
$$ i = \res(b_1),\ \ j = \res(b_2),\ \ Y_0 = \sh(T_{< k}). $$
If $b_1$ (resp.\ $b_2$) is in the $p$th row (resp.\ the $q$th row) of $T$, then we have
$$
\frac{\qF_q(T')}{\qF_q(T)} = \frac{\qE_q(T')}{\qE_q(T)}  = \left\{
                           \begin{array}{ll}
                             q^{-(\alpha_i, \alpha_j)} \left( \displaystyle \frac{1-(-q^2)^{l(b_1;Y_0)+1}}{1-(-q^2)^{l(b_1;Y_0)}} \right)  & \hbox{ if } \res(b_1) = \res(b_2),\ p = q +1, \\
                             q^{-(\alpha_i, \alpha_j)} & \hbox{ otherwise.}
                           \end{array}
                         \right.
$$
\end{Prop}
\begin{proof}
Since $T < T'$, $b_2$ appears to the left of $b_1$ in $T$ by $\eqref{Eq: going up in weak order}$.
Note that, when the case $\res(b_1) = \res(b_2) $ and $ p = q +1$, it can be viewed pictorially as follows.
$$
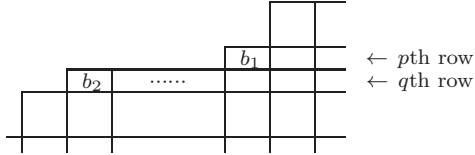
\xy
(33,9)*{};(43,9)*{} **\dir{-};
(27,3)*{};(43,3)*{} **\dir{-};
(6,0)*{};(43,0)*{} **\dir{-};
(0,-3)*{};(43,-3)*{} **\dir{-};
(-2,-9)*{};(43,-9)*{} **\dir{-};
(39,9)*{};(39,-11)*{} **\dir{-};
(33,9)*{};(33,-11)*{} **\dir{-};
(27,3)*{};(27,-11)*{} **\dir{-};
(12,0)*{};(12,-11)*{} **\dir{-};
(6,0)*{};(6,-11)*{} **\dir{-};
(0,-3)*{};(0,-11)*{} **\dir{-};
(19.5,-1.6)*{ _{\cdots \cdots}  };
(9.5,-1.6)*{ _{b_2}  }; (53,-1.6)*{  _{ \leftarrow \ q\text{th row}  }   };
(30.5,1.4)*{ _{b_1}  }; (53,1.4)*{ _{ \leftarrow\  p\text{th row}  }   };
\endxy
$$

For $t=1, \ldots , n$, let
\begin{align*}
Y_T(t) &= \sh(T_{\le t}), \quad b_T(t) = \text{ the block with entry $t$ in $T$}, \\
\quad Y_{T'}(t) &= \sh(T'_{\le t}), \quad b_{T'}(t) = \text{ the block with entry $t$ in $T'$}.
\end{align*}
Then, the condition $s_k T = T'$ tells that $b_T(t)$ and $b_{T'}(t)$ are placed at the same position in $Y$ if $t \ne k, k+1$, and $Y_T(t) = Y_{T'}(t)$ if $t \ne k$.
Set
$$ Y_1 = Y_T(k) , \quad Y_2 = Y_{T'}(k), \quad Y_3 = Y_{T}(k+1). $$

It follows from $\eqref{Eq: Laurent poly}$ that
$$
\frac{\qF_q(T')}{\qF_q(T)} = \frac{ \displaystyle q_j^{L_j(b_2;Y_0)} \left( \frac{1-(-q^2)^{l(b_2;Y_0)}}{q} \right)^{  \delta_-  (b_2;Y_0)}
q_i^{L_i(b_1;Y_2)} \left( \frac{1-(-q^2)^{l(b_1;Y_2)}}{q} \right)^{  \delta_-  (b_1;Y_2)} }
{ \displaystyle q_i^{L_i(b_1;Y_0)} \left( \frac{1-(-q^2)^{l(b_1;Y_0)}}{q} \right)^{  \delta_- (b_1;Y_0)}
q_j^{L_j(b_2;Y_1)} \left( \frac{1-(-q^2)^{l(b_2;Y_1)}}{q} \right)^{  \delta_-  (b_2;Y_1)} }
.$$
It follows from $\eqref{Eq: di}$ and the fact that $b_2$ appears to the left of $b_1$ in $T$ that
\begin{align*}
& q_j^{L_j(b_2;Y_0)} = q_j^{L_j(b_2;Y_1)}, \ \  \quad q_i^{-\alpha_j(h_i)} q_i^{L_i(b_1;Y_0)} = q_i^{L_i(b_1;Y_2)} \\
 & \delta_- (b_2;Y_0) = \delta_- (b_2;Y_1), \quad   \delta_- (b_1;Y_0) = \delta_- (b_1;Y_2),  \\
&  l(b_2;Y_0) = l(b_2; Y_1), \\
& l(b_1;Y_2) = \left\{
                  \begin{array}{ll}
                    l(b_1;Y_0)+1 & \hbox{ if } \res(b_1) = \res(b_2),\ p = q +1,  \\
                    l(b_1;Y_0) & \hbox{ otherwise,}
                  \end{array}
                \right.
\end{align*}
which yields
$$
\frac{\qF_q(T')}{\qF_q(T)} = \left\{
                           \begin{array}{ll}
                             q_i^{-\alpha_j(h_i)} \displaystyle \left( \frac{1-(-q^2)^{l(b_1;Y_0)+1}}{1-(-q^2)^{l(b_1;Y_0)}}\right) & \hbox{ if } \res(b_1) = \res(b_2),\ p = q +1, \\
                             q_i^{-\alpha_j(h_i)} & \hbox{ otherwise.}
                           \end{array}
                         \right.
$$

In the same manner,  $\eqref{Eq: Laurent poly}$ gives
$$
\frac{\qE_q(T')}{\qE_q(T)} = \frac{ \displaystyle q_j^{-R_j(b_2;Y_2)} \left( \frac{1-(-q^2)^{l(b_2;Y_2)}}{q} \right)^{ \delta_+ (b_2;Y_2)}
q_i^{-R_i(b_1;Y_3)} \left( \frac{1-(-q^2)^{l(b_1;Y_3)}}{q} \right)^{  \delta_+  (b_1;Y_3)} }
{ \displaystyle q_i^{-R_i(b_1;Y_1)} \left( \frac{1-(-q^2)^{l(b_1;Y_1)}}{q} \right)^{ \delta_+ (b_1;Y_1)}
q_j^{-R_j(b_2;Y_3)} \left( \frac{1-(-q^2)^{l(b_2;Y_3)}}{q} \right)^{ \delta_+ (b_2;Y_3)} }
.$$
It follows from $\eqref{Eq: di}$ and the fact that $b_2$ appears to the left of $b_1$ in $T$ that
\begin{align*}
& q_j^{\alpha_i(h_j)} q_j^{-R_j(b_2;Y_2)} = q_j^{-R_j(b_2;Y_3)}, \ \  \quad q_i^{-R_i(b_1;Y_1)} = q_i^{-R_i(b_1;Y_3)} \\
 & \delta_+(b_2;Y_3) = \delta_+(b_2;Y_2), \quad  \delta_+(b_1;Y_1) = \delta_+(b_1;Y_3),  \\
&  l(b_1;Y_1) = l(b_1; Y_3), \\
& l(b_2;Y_2) = \left\{
                  \begin{array}{ll}
                    l(b_2;Y_3)+1 & \hbox{ if } \res(b_1) = \res(b_2),\ p = q +1,  \\
                    l(b_2;Y_3) & \hbox{ otherwise.}
                  \end{array}
                \right.
\end{align*}
Since $l(b_1;Y_0) = l(b_2;Y_3)$ when $\res(b_1) = \res(b_2)$ and $ p = q +1$, we have
$$
\frac{\qE_q(T')}{\qE_q(T)} = \left\{
                           \begin{array}{ll}
                             q_j^{-\alpha_i(h_j)} \displaystyle \left( \frac{1-(-q^2)^{l(b_1;Y_0)+1}}{1-(-q^2)^{l(b_1;Y_0)}}\right) & \hbox{ if } \res(b_1) = \res(b_2),\ p = q +1, \\
                             q_j^{-\alpha_i(h_j)} & \hbox{ otherwise.}
                           \end{array}
                         \right.
$$
\end{proof}

For $t \in \Z_{\ge 0}$ and $\underline{m}=(m_1,m_2,\ldots,m_n) \in \Z_{\ge 0}^n$, we define
\begin{align*}
\langle t \rangle = \prod_{k=1}^t (1-(-q^2)^k)
\quad \text{ and }
\quad \langle \underline{m} \rangle = \prod_{k=1}^n \langle m_k
\rangle.
\end{align*}
The following corollary tells us that $\dfrac{ \qE_q(T)}{\qF_q(T)} $ is
a Laurent polynomial depending only on the proper Young wall $\sh(T)$.

\begin{Cor} \label{Cor: ratio E and F}
Let $Y$ be a proper Young wall. For any $T \in \ST(Y)$, we obtain
$$\dfrac{ \qE_q(T)}{\qF_q(T)} = q^{ - \df(\Lambda_0,\Lambda_0 - \wt(Y))} \langle \hst_Y \rangle .$$
where $ \hst_Y$ is given as $\eqref{Eq: hst}$.
\end{Cor}
\begin{proof}

It follows from Theorem \ref{Thm: weak order} and Proposition \ref{Prop: ratio for E, F} that
$$ \dfrac{ \qE_q(T)}{\qF_q(T)} = \dfrac{ \qE_q(T^Y)}{\qF_q(T^Y)}. $$
Since $c_+(Y) - c_-(Y) = |\hst_Y| $, the assertion follows from Lemma \ref{Lem: def+} and Theorem \ref{Thm: qF over qE}.
\end{proof}

Consider a proper Young wall $Y$ with $n$ blocks. For $T \in \ST(Y)$, set
$$\qm_q(T) = \qF_q(T) / \qF_q(T^Y).$$
Let $w = w(T) \in \sg_n $ and write its reduced expression $w = s_{i_t} \ldots s_{i_2} s_{i_1}$,
and set $T_0 = T^Y$ and $T_k = s_{i_k} T_{k-1}$ ($k=1,\ldots, t$). Note that $T = w T^Y$ and $T_{k-1} < T_{k}$.
Since
$$ \qm_q(T) =  \qF_q(T) / \qF_q(T^Y) = \prod_{k=1}^t \qF_q(T_k) / \qF_q(T_{k-1}),$$
$\qm_q(T)$ may be computed directly by considering Theorem \ref{Thm: qF over qE} and Proposition \ref{Prop: ratio for E, F}.
Moreover, Proposition \ref{Prop: ratio for E, F} tells
\begin{align} \label{Eq: qE, qF, qm}
\qm_q(T) =  \prod_{k=1}^t \qF_q(T_k) / \qF_q(T_{k-1}) =  \prod_{k=1}^t \qE_q(T_k) / \qE_q(T_{k-1}) =  \qE_q(T) / \qE_q(T^Y).
\end{align}

We define
\begin{align} \label{Eq: def of m}
\qm_q(Y) = \sum_{T \in \ST(Y)} \qm_q(T).
\end{align}
It follows from $\eqref{Eq: qE, qF, qm}$ that
\begin{equation} \label{Eq: mq}
\begin{aligned}
\qF_q(T^Y) \qm_q(Y) &= \sum_{T \in \ST(Y)} \qF_q(T), \quad \ \qE_q(T^Y)
\qm_q(Y) &= \sum_{T \in \ST(Y)} \qE_q(T).
\end{aligned}
\end{equation}

\begin{Prop} \label{Prop: m(Y)}
For any $T \in \ST(Y)$, $ \qm_q(T)$ is a Laurent polynomial and hence so is $\qm_q(Y)$.
\end{Prop}
\begin{proof}

Since $1- (-t)^k$ is divided by $1+t$ for all $k \in \Z_{>0}$, we have
$$\dfrac{ 1-(-q^2)^{l(b;T)+1} }{1+q^2}\in \Z[q]$$
for all $b \in C_-(Y)$. Thus it follows from Theorem \ref{Thm: qF over qE} that
$ \qm_q(T)$ is a Laurent polynomial.
\end{proof}







\vskip 1em

\section{Graded dimension formulas } \label{Sec: dimension formula}


\subsection{Graded dimension formulas} \label{Sec: graded dim fomula2}
Recall the Fock space representation $\mathcal{F}(\Lambda_0)$ of
type $\tyX$ given in Section \ref{Sec: Fock}.

\begin{defn}
For a standard tableau $T \in \ST(Y)$, we define the {\it residue
sequence} of $T$ by
$$ \res(T) = (\res_1(T), \res_2(T), \ldots, \res_n(T)) \in I^n, $$
where $\res_k(T)$ is the residue of the block with entry $k$ in $T$.
\end{defn}

For $\nu \in I^n$, we set
$$ \ST(Y, \nu) = \{ T\in \ST(Y) \mid \res(T) = \nu  \} . $$
The action $\eqref{Eq: action of f, e on Fock space}$ of Chevalley
generators $f_i, e_i$ of the Fock space $\mathcal{F}(\Lambda_0)$
gives the following lemma.

\begin{Lem} \label{Lem: graded dim}
Let $Y \in \mathcal{Z}(\Lambda_0)$ be a proper Young wall with $n$ blocks and let $\nu = (\nu_1, \nu_2, \ldots, \nu_n ) \in I^\beta $.
Then we obtain
\begin{align*}
e_{\nu_1} e_{\nu_2} \ldots e_{\nu_n} Y &= \left( \sum_{T \in \ST(Y, \nu)} \qE_q(T) \right) \emptyset, \\
f_{\nu_n} f_{\nu_{n-1}} \ldots f_{\nu_1} \emptyset &= \sum_{Y' \in \mathcal{Z}(\Lambda_0),\ \wt(Y') = \Lambda_0 - \beta} \left( \sum_{T \in \ST(Y', \nu)} \qF_q(T) \right) Y',
\end{align*}
where $\qE_q(T)$ and $\qF_q(T)$ are the Laurent polynomials defined by $\eqref{Eq: Laurent poly}$.
\end{Lem}

We are now ready to prove our main theorem, graded dimension formulas for $\fqH(\beta)$.
\begin{Thm} \label{Thm: q-dim formulas}
Let $\beta \in \rlQ^+$. For $\nu, \nu' \in I^\beta$, we have
\begin{align*}
\dim_q e(\nu) \fqH(\beta) e(\nu') &= q^{\df (\Lambda_0, \beta)} \sum_{Y\in \mathcal{Z}(\Lambda_0)} \sum_{T\in \ST(Y, \nu),\ T' \in \ST(Y, \nu')} \qE_q(T) \qF_q(T'),\\
&=  \sum_{Y\in \mathcal{Z}(\Lambda_0)}
  \langle \hst_Y \rangle \qF_q(T^Y) ^2  \sum_{T\in \ST(Y, \nu),\ T' \in
\ST(Y, \nu')} \qm_q(T) \qm_q(T') .
\end{align*}
Moreover, we obtain
\begin{align*}
\dim_q  \fqH(\beta) &=  \sum_{Y\in \mathcal{Z}(\Lambda_0) ,\ \wt(Y) = \Lambda_0 - \beta }
 \langle \hst_Y \rangle \qF_q(T^Y) ^2 \   \qm_q(Y)^2.
\end{align*}
where $ \hst_Y $ is given as $\eqref{Eq: hst}$, and $\qE_q(T)$, $\qF_q(T)$ and $\qm_q(Y)$ are the Laurent polynomials defined in Section \ref{Sec: Laurent poly}.
\end{Thm}
\begin{proof}
Let
$$ \nu = (\nu_1, \nu_2, \ldots, \nu_n),\quad \nu' = (\nu'_1, \nu'_2, \ldots, \nu'_n).  $$
By Theorem \ref{Thm: Fock space}, the submodule of
$\mathcal{F}(\Lambda_0)$ generated by $\emptyset$ is isomorphic to
the irreducible highest weight module $V(\Lambda_0)$. Thus, it
follows from Proposition \ref{Prop: graded dim}, Corollary \ref{Cor: ratio E and F} and Lemma \ref{Lem: graded dim} that
\begin{align*}
\left( \dim_q e(\nu) \fqH(\beta) e(\nu') \right) \emptyset &= q^{\df(\Lambda_0,\beta)} e_{\nu_1} \cdots e_{\nu_n} f_{\nu'_n} \cdots f_{\nu'_1} \emptyset \\
&=  q^{\df(\Lambda_0, \beta)} e_{\nu_1} \cdots e_{\nu_n} \sum_{Y\in \mathcal{Z}(\Lambda_0)} \left( \sum_{T' \in \ST(Y, \nu')} \qF_q(T') \right) Y \\
&= q^{\df (\Lambda_0, \beta)} \sum_{Y\in \mathcal{Z}(\Lambda_0)}
\sum_{ \substack{T\in \ST(Y, \nu),\\ T' \in \ST(Y, \nu')}} \qE_q(T)
\qF_q(T') \
\emptyset \\
&=  \sum_{Y\in \mathcal{Z}(\Lambda_0)} \langle \hst_Y \rangle   \sum_{\substack{T\in \ST(Y, \nu),\\ T' \in \ST(Y, \nu')}}
\qF_q(T) \qF_q(T') \ \emptyset, \\
&=  \sum_{Y\in \mathcal{Z}(\Lambda_0)}
  \langle \hst_Y \rangle \qF_q(T^Y) ^2   \sum_{T\in \ST(Y, \nu),\ T' \in
\ST(Y, \nu')} \qm_q(T) \qm_q(T') \ \emptyset,
\end{align*}
which yields the first assertion.

On the other hand, it follows from $\eqref{Eq: mq}$ and $ \fqH(\beta) = \bigoplus_{\nu, \nu' \in I^\beta} e(\nu) \fqH(\beta) e(\nu') $ that
\begin{align*}
\dim_q \fqH(\beta) &= \dim_q \bigoplus_{\nu, \nu' \in I^\beta}
e(\nu) \fqH(\beta) e(\nu')  = q^{\df (\Lambda_0, \beta)} \sum_{\nu,
\nu' \in I^\beta} \sum_{Y\in \mathcal{Z}(\Lambda_0)}
\sum_{ \substack{ T\in \ST(Y, \nu), \\ \ T' \in \ST(Y, \nu')}} \qE_q(T) \qF_q(T') \\
&= q^{\df (\Lambda_0, \beta)} \sum_{Y\in \mathcal{Z}(\Lambda_0),\ \wt(Y) = \Lambda_0 - \beta} \left( \sum_{T\in \ST(Y)} \qE_q(T) \right) \left( \sum_{T' \in \ST(Y)}  \qF_q(T') \right)
\allowdisplaybreaks \\
&= q^{\df (\Lambda_0, \beta)} \sum_{Y\in \mathcal{Z}(\Lambda_0),\
\wt(Y) = \Lambda_0 - \beta} \qE_q(T^Y)  \qF_q(T^Y) \qm_q(Y)^2 \\
&=  \sum_{Y\in \mathcal{Z}(\Lambda_0) ,\ \wt(Y) = \Lambda_0 - \beta }
  \langle \hst_Y \rangle \qF_q(T^Y)^2  \   \qm_q(Y)^2,
\end{align*}
which complete the proof.
\end{proof}

\subsection{Forgetting the grading }
In this subsection, we study the dimensions for $e(\nu) \fqH(\beta)
e(\nu')$ and  $ \fqH(\beta)$ which were investigated in
\cite{AP12,AP13}.

\medskip

Recall the subset $\ST_\infty(Y) \subset \ST(Y)$ defined by $\eqref{Eq: def of STinf}$.
Let
$$ \ST_\infty(Y, \nu) = \{ T\in \ST_\infty(Y) \mid \res(T) = \nu  \}. $$
The lemmas below follow immediately from the definition of $\ST_\infty(Y)$ and Theorem \ref{Thm: qF over qE}.


\begin{Lem} \label{Lem: ST inf}
For every $T \in \ST(Y) \setminus \ST_\infty(Y)$, there exist a
block $b \in C_+(Y)$ with $l(b;T)>0$. Moreover, if $\lambda_T$ is a
strict partition, there exist also a block $b' \in C_-(Y)$ with
$l(b';T)>0$.
\end{Lem}

\begin{Lem} \label{Lem: evaluation}
For a standard tableau $T \in \ST(Y)$, we have the following:
\begin{enumerate}
\item $\qE_q(T)|_{q=1} = \left\{
                           \begin{array}{ll}
                             2^{c_+(Y)} & \hbox{ if there is no $b \in C_+(Y)$ with $l(b;T)>0$} , \\
                             0 & \hbox{ otherwise.}
                           \end{array}
                         \right.
$
\item $\qF_q(T)|_{q=1} = \left\{
                           \begin{array}{ll}
                             2^{c_-(Y)} & \hbox{ if there is no $b \in C_-(Y)$ with $l(b;T)>0$}, \\
                             0 & \hbox{ otherwise.}
                           \end{array}
                         \right.
$
\end{enumerate}
\end{Lem}
%
%
%
%

We take an element $d_\tyX \in \wlP^\vee\otimes_\Z \Q $ such that
$$
\langle d_\tyX , \Lambda_0 \rangle = 0, \quad \langle d_\tyX , \alpha_i \rangle = \left\{
                                                                                    \begin{array}{ll}
                                                                                      1 & \hbox{ if } (i = 0,\ \tyX = A_{2 \ell}^{(2)}) \text{ or }  (i = 0,\ell,\ \tyX = D_{ \ell+1}^{(2)}),  \\
                                                                                      0 & \hbox{otherwise.}
                                                                                    \end{array}
                                                                                  \right.
$$
Note that, for a proper Young wall $Y$, $-\langle d_\tyX, \wt(Y)  \rangle $ is the number of half-unit height blocks in $Y$.
One can easily check that
$$ -\langle d_\tyX, \wt(Y)  \rangle - \ell(Y) = c_-(Y)+c_+(Y),$$
where $\ell(Y)$ is the number of columns in $Y$. Then, Theorem
\ref{Thm: q-dim formulas}, Lemma \ref{Lem: ST inf} and Lemma
\ref{Lem: evaluation} recover the dimension formulas given in
\cite[Thm.\ 3.4]{AP12} and \cite[Thm.\ 3.2]{AP13} by evaluating the
graded dimension formulas given in Theorem \ref{Thm: q-dim formulas}
at $q=1$.

\begin{Thm} \label{Thm: dim formulas}
For $\beta \in \rlQ^+$ and $\nu, \nu' \in I^\beta$, we have
\begin{align*}
\dim e(\nu) \fqH(\beta) e(\nu') &= \sum_{Y\in
\mathcal{Z}(\Lambda_0)} \sum_{ \substack{T\in \ST_\infty(Y, \nu),\\
T' \in \ST_\infty(Y, \nu')}} 2^{ - \langle d_\tyX, \wt(Y)  \rangle -\ell(Y)} |\ST_{\infty}(Y,\nu)||\ST_{\infty}(Y,\nu')|, \\
\dim  \fqH(\beta) &=  \sum_{Y\in \mathcal{Z}(\Lambda_0),\ \wt(Y) =
\Lambda_0 - \beta} 2^{ - \langle d_\tyX, \wt(Y)  \rangle -\ell(Y)} |\ST_{\infty}(Y)|^2.
\end{align*}
\end{Thm}

\vskip 2em

\noindent
{\bf Acknowledgments.} The authors would like to thank the anonymous reviewers for their valuable comments and suggestions.

\vskip 2em


\bibliographystyle{amsplain}


\end{document}